\documentclass[12pt,reqno]{amsart}

\usepackage{times}

\usepackage[T1]{fontenc}

\usepackage{mathrsfs}

\usepackage{latexsym}

\usepackage{graphics}

\usepackage{graphicx}

\usepackage{epstopdf}

\usepackage{amsmath,amsfonts,amsthm,amssymb,amscd}

\input amssym.def

\input amssym.tex

\usepackage{color}

\usepackage{nicefrac}

\newcommand\be{\begin{equation}}

\newcommand\ee{\end{equation}}

\newcommand\bea{\begin{eqnarray}}

\newcommand\eea{\end{eqnarray}}

\newcommand\bi{\begin{itemize}}

\newcommand\ei{\end{itemize}}

\newcommand\ben{\begin{enumerate}}

\newcommand\een{\end{enumerate}}

\newtheorem{thm}{Theorem}[section]

\newtheorem{cor}[thm]{Corollary}

\newtheorem{lem}[thm]{Lemma}

\newtheorem{defi}[thm]{Definition}

\newtheorem{rek}[thm]{Remark}

\newcommand{\twocase}[5]{#1 \begin{cases} #2 & \text{#3}\\ #4
&\text{#5} \end{cases}   }
\newcommand{\threecase}[7]{#1 \begin{cases} #2 &
\text{#3}\\ #4 &\text{#5}\\ #6 &\text{#7} \end{cases}   }

\newcommand{\E}{\mathbb{E}}

\newcommand{\mattwo}[4]
{\left(\begin{array}{cc}
                        #1  & #2   \\
                        #3 &  #4
                          \end{array}\right) }

\newcommand{\intii}{\int_{-\infty}^\infty}

\newcommand{\ncr}[2]{\binom{#1}{#2}}
\newcommand{\ccr}[2]{{\rm Cr}_{#1, #2}}

\numberwithin{equation}{section}

\begin{document}

\title{Distribution of Eigenvalues of Weighted, Structured Matrix Ensembles}

\author{Olivia Beckwith}
\email{obeckwith@gmail.com}
\address{Department of Mathematics, Harvey Mudd College, Claremont, CA 91711}

\author{Victor Luo}
\email{vdl1@williams.edu}
\address{Department of Mathematics and Statistics, Williams College, Williamstown, MA 01267}

\author{Steven J. Miller}
\email{sjm1@williams.edu, Steven.Miller.MC.96@aya.yale.edu} \address{Department of Mathematics and Statistics, Williams College, Williamstown, MA 01267}

\author{Karen Shen}\email{shenk@stanford.edu}
\address{Department of Mathematics, Stanford University, Stanford, CA 94305}

\author{Nicholas Triantafillou}
\email{ngtriant@umich.edu}
\address{Department of Mathematics, University of Michigan, Ann Arbor, MI 48109}



\date{\today}

\thanks{We thank Colin Adams, Arup Bose, Satyan Devadoss, Allison Henrich, Murat Kolo$\check{{\rm g}}$lu and Gene Kopp for helpful conversations, and the referee of an earlier draft of the paper for several useful comments. The first and third named authors were supported by NSF grant DMS0850577 and Williams College; the second, fourth and fifth named authors were partially supported by NSF grant DMS0970067.}

\begin{abstract}
The study of the limiting distribution of eigenvalues of $N \times N$ random matrices as $N \rightarrow \infty$ has many applications, including nuclear physics, number theory and network theory. One of the most studied ensembles is that of real symmetric matrices with independent entries drawn from identically distributed nice random variables, where the limiting rescaled spectral measure is the semi-circle. Studies have also determined the limiting rescaled spectral measures for many structured ensembles, such as Toeplitz and circulant matrices. These systems have very different behavior; the limiting rescaled spectral measures for both have unbounded support. Given a structured ensemble such that (i) each random variable occurs $o(N)$ times in each row of matrices in the ensemble and (ii) the limiting rescaled spectral measure $\widetilde{\mu}$ exists, we introduce a parameter to continuously interpolate between these two behaviors. We fix a $p \in [1/2, 1]$ and study the ensemble of signed structured matrices by multiplying the $(i,j)^{\text{th}}$ and $(j,i)^{\text{th}}$ entries of a matrix by a randomly chosen $\epsilon_{ij} \in \{1, -1\}$, with ${\rm Prob}(\epsilon_{ij} = 1)  = p$ (i.e., the Hadamard product). For $p = 1/2$ we prove that the limiting signed rescaled spectral measure is the semi-circle. For all other $p$, we prove the limiting measure has bounded (resp., unbounded) support if $\widetilde{\mu}$ has bounded (resp., unbounded) support, and converges to $\widetilde{\mu}$ as $p \to 1$. Notably, these results hold for Toeplitz and circulant matrix ensembles.

The proofs are by Markov's Method of Moments. The analysis of the $2k^{\text{th}}$ moment for such distributions involves the pairings of $2k$ vertices on a circle. The contribution of each pairing in the signed case is weighted by a factor depending on $p$ and the number of vertices involved in at least one crossing. These numbers are of interest in their own right, appearing in problems in combinatorics and knot theory. The number of configurations with no vertices involved in a crossing is well-studied, and are the Catalan numbers. We discover and prove similar formulas for configurations with $4, 6, 8$ and $10$ vertices in at least one crossing. We derive a closed-form expression for the expected value and determine the asymptotics for the variance for the number of vertices in at least one crossing. As the variance converges to 4, these results allow us to deduce properties of the limiting measure.
\end{abstract}

\maketitle

\tableofcontents




\section{Introduction}

\subsection{Background}

Though Random Matrix Theory began with statistics investigations by Wishart \cite{Wis}, it was through the work of Wigner \cite{Wig1,Wig2,Wig3,Wig4,Wig5}, Dyson \cite{Dy1,Dy2} and others that its true power and universality became apparent. Wigner's great insight was that ensembles of matrices with randomly chosen entries model well many nuclear phenomena. For example, in quantum mechanics the fundamental equation is $H \Psi_n = E_n \Psi_n$ ($H$ is the Hamiltonian, $\Psi_n$ the energy eigenstate with eigenvalue $E_n$). Though $H$ is too complicated to diagonalize, a typical $H$ behaves similarly to the average behavior of the ensemble of matrices where each independent entry is chosen independently from some fixed probability distribution. Depending on the physical system, the matrix $H$ is constrained. The most common $H$ is real-symmetric (where the limiting rescaled spectral measure is the semi-circle) or Hermitian. In addition to physics, these matrix ensembles successfully model diverse fields from number theory \cite{ILS,KS1,KS2,KeSn,Mon,RS} to random graphs \cite{JMRR,MNS} to bus routes in Mexico \cite{BBDS,KrSe}.

The original ensembles studied had independent entries chosen from a fixed probability distribution with mean 0, variance 1 and finite higher moments. For such ensembles, the limiting rescaled spectral measure could often be computed, though only recently (see \cite{ERSY,ESY,TV1,TV2}) was the limiting spacing measure between normalized eigenvalues determined for general distributions. See \cite{Fo,Meh} for a general introduction to Random Matrix Theory, and \cite{Dy3,FM,Hay} for a partial history.

Recently there has been much interest in studying highly structured sub-ensembles of the family of real symmetric matrices, where new limiting behavior emerges. Examples include band matrices, circulant matrices, random abelian $G$-circulant matrices, adjacency matrices associated to $d$-regular graphs, and Hankel and Toeplitz matrices, among others \cite{BasBo2,BasBo1,BanBo,BCG,BHS1,BHS2,BM,BDJ,GKMN,HM,JMP,Kar,KKMSX,LW,MMS,McK,Me,Sch}. Two particularly interesting cases are the Toeplitz \cite{BDJ,HM} and singly palindromic Toeplitz ensemble \cite{MMS}, which we now generalize (though our arguments would follow through with only minor changes for other structured ensembles). A real symmetric Toeplitz matrix is constant along its diagonals, while its palindromic variant has the additional property that its first row is a palindrome. The limiting rescaled spectral measures of these ensembles have been proven to exist; it is the Gaussian in the singly palindromic case, and almost a Gaussian in the Toeplitz case (the limiting  rescaled spectral measure has unbounded support, though the moments grow significantly slower than the Gaussian's).

As these matrices are small sub-families of the family of all real symmetric matrices, it is not surprising that new behavior is seen. A natural question to ask is whether or not there is a way to `fatten' these ensembles and regain the behavior of the full real symmetric ensemble. This is similar to what happens for the adjacency matrices of $d$-regular graphs. For fixed $d$ the limiting rescaled spectral measure is Kesten's measure \cite{McK}, which converges as $d\to\infty$ to the semi-circle (see \cite{GKMN} for the related problem of the limiting rescaled spectral measure of weighted $d$-regular graphs). We can ask similar questions about band matrices, and again see a transition in behavior as a parameter grows \cite{Sch}.

Before stating our results, we first quickly review some standard notation (see for example \cite{HM,JMP,KKMSX,MMS}).

\begin{itemize}

\item Random matrix ensemble: In this paper a random matrix ensemble is a collection of $N \times N$ (with $N \to\infty$) real symmetric matrices whose independent entries are drawn from identically distributed random variables whose density $\mathfrak{p}$ has mean 0, variance 1 and finite higher moments. We often study structured ensembles, where there are additional relations beyond the requirement of  being real symmetric. The probability measure attached to the $N\times N$ matrices in the ensemble is \be {\rm Prob}(A) dA \ = \ \prod_{(i,j) \in \mathcal{I}_N} \mathfrak{p}(a_{ij}) da_{ij}, \ee where $\mathcal{I}_N$ is a complete set of indices corresponding to the independent entries of our $N\times N$ matrices. For example, for real symmetric Toeplitz matrices the only dependency condition is that $a_{ij} = a_{k\ell}$ if $|i-j| = |k-\ell|$, and we may thus take $\mathcal{I}_N = \{a_{11}, a_{12}, \dots, a_{N1}\}$.

\item Empirical spectral measure: Given an $N\times N$ real symmetric matrix $A$, its empirical spectral measure is \be \mu_A(x) \ = \ \frac1{N} \sum_{k=1}^N \delta(x - \lambda_k(A)), \ee with $\delta(x)$ the Dirac delta functional and the $\lambda_k(A)$'s are the eigenvalues of $A$.

\item Rescaled empirical spectral measure: The rescaled empirical spectral measure of $A$, denoted $\widetilde{\mu}_A(x)$, is \be \widetilde{\mu}_A(x) \ = \ \frac1{N} \sum_{k=1}^N \delta\left(x - \frac{\lambda_k(A)}{\mathfrak{c} N^r}\right); \ee notice \be \widetilde{\mu}_A(x) \ = \ \mu_{A/\mathfrak{c}N^r}(x).\ee Typically the $A$'s are chosen from a random matrix ensemble, and we have one pair $(\mathfrak{c}, r)$ for all the $A$'s. In this paper usually $r = 1/2$ (this is a consequence of the eigenvalue trace lemma and the central limit theorem) as our random matrix ensembles are full (i.e., each entry is drawn from a random variable with mean 0 and variance 1). The situation would be drastically different if we considered matrices where many entries are forced to be zero, such as the adjacency matrices associated to $d$-regular graphs (where $r=0$ as the eigenvalues do not grow with $N$) or band matrices where the band width is small relative to $N$.

\item Hadamard products: Given real symmetric $N \times N$ matrices $A = (a_{ij})$ and $B = (b_{ij})$, their Hadamard product, denoted $A \circ B$, is the matrix whose $(i,j)$\textsuperscript{th} entry is $a_{ij} b_{ij}$; its empirical spectral measure is $\mu_{A \circ B}(x)$.

\item Limiting spectral measure: If the limit of the sequence of average moments of a random matrix ensemble, \be \lim_{N\to\infty} \intii \cdots \intii x^k \widetilde{\mu}_A(x) {\rm Prob}(A) dA \ \ \ (k\ {\rm a\ positive\ integer}), \ee exists and uniquely determines a measure, that measure is called the limiting spectral measure of the ensemble.

\item Limiting signed rescaled spectral measure: Let $p \in [1/2, 1]$ and consider the random matrix ensemble of real symmetric matrices $\mathcal{E} = (\epsilon_{ij})$ with the independent entries independent identically distributed random variables that are 1 with probability $p$ and -1 with probability $1-p$; we call this the \emph{signed} or \emph{weighted} ensemble. Given a random matrix ensemble with matrices $A$, consider the signed random matrix ensemble with matrices $A \circ \mathcal{E}$. The ensemble has measure \be \left(\prod_{i \le j} p^{(1+\epsilon_{ij})/2} (1-p)^{(1-\epsilon_{ij})/2}\right)  {\rm Prod}(A) dA. \ee We rescale the eigenvalues of the Hadamard product by the same factor we used for the unsigned matrices; thus  \be \widetilde{\mu}_{A \circ \mathcal{E}}(x) \ = \ \mu_{(A/\mathfrak{c}N^r) \circ \mathcal{E}}(x).\ee The average $k$\textsuperscript{th} moment is \be \intii \cdots \intii \prod_{1 \le i \le j \le N} \sum_{\epsilon_{ij} \in \{-1,1\}} \int_{x=-\infty}^\infty x^k \widetilde{\mu}_{A\circ \mathcal{E}}(x) p^{(1+\epsilon_{ij})/2} (1-p)^{(1-\epsilon_{ij})/2} {\rm Prod}(A) dx\ dA.\ee

\end{itemize}

\ \\

The key to our analysis is the Eigenvalue Trace Lemma, which implies that the $k$\textsuperscript{th} moment of $\widetilde{\mu}_A$ is \be M_{k;N}(A) \ = \ \intii x^k \widetilde{\mu}_A(x) dx \ = \ \frac{{\rm Trace}(A^k)}{\mathfrak{c}^k N^{rk+1}}. \ee The advantage of this formulation is that we convert what we want to study (the eigenvalues) to something we understand (the matrix entries, which are randomly chosen). We now integrate the above over the family, reducing the computation to averaging polynomials of the matrix elements over the family. Determining the answer frequently involves solving difficult combinatorial problems to count the number of configurations with a given contribution, with the structure of the ensemble determining the combinatorics.

We concentrate on the family of highly palindromic real symmetric Toeplitz matrices, introduced in \cite{JMP} and defined below, for several reasons. This is a well-studied family, with certain special cases corresponding to some of the more important classical ensembles. Further, the structure of these matrices is conducive to obtaining tractable closed form expressions for many of the quantities. It is straightforward to generalize these results to other structured ensembles whose limiting rescaled spectral measure exists, and we sketch the proof.

\begin{defi}
For fixed $n$, a \textbf{(degree $n$) $N\times N$ highly palindromic real symmetric Toeplitz matrix} is one in which the first row is $2^n$ copies of a palindrome, where the entries are iidrv whose density $\mathfrak{p}$ has mean 0, variance 1 and finite higher moments; for brevity we often omit ``real symmetric'' below. We always assume $N$ to be a multiple of $2^n$ so that each element occurs exactly $2^{n+1}$ times in the first row. If $n=0$ we say it is a \textbf{singly palindromic Toeplitz matrix}. If $a_{ij}$ is the entry in the $i$\textsuperscript{th} row and $j$\textsuperscript{th} column of $A$, then we set $b_{|i-j|} = a_{ij}$ (if the ensemble is at least doubly palindromic, then the $b$'s are not distinct and satisfy additional relations due to the palindromicity). For example, a doubly palindromic Toeplitz matrix is of the form \be\label{eq:defrsptmat} A_N\ =\ \left(
\begin{array}{cccccccccc}
b_0  &  b_1  &  \cdots  & b_1 & b_0 & b_0 & b_1 & \cdots  & b_1   & b_0  \\
b_1  &  b_0  &  \cdots  & b_2 & b_1 & b_0 & b_0 & \cdots  & b_2   & b_1  \\
b_2  &  b_1  &  \cdots  & b_3 & b_2 & b_1 & b_0 & \cdots  & b_3   & b_2  \\
\vdots  &  \vdots  &  \ddots  & \vdots & \vdots & \vdots & \vdots & \ddots  & \vdots   & \vdots  \\
b_2  &  b_3  &  \cdots  & b_0 & b_1 & b_2 & b_3 & \cdots  & b_1   & b_2  \\
b_1  &  b_2  &  \cdots  & b_0 & b_0 & b_1 & b_2 & \cdots  & b_0   & b_1  \\
b_0  &  b_1  &  \cdots  & b_1 & b_0 & b_0 & b_1 & \cdots  & b_1   & b_0  \\
\end{array}\right). \nonumber \ee

\noindent The entries of the matrices are constant along diagonals. Furthermore, entries on two diagonals that are $N/2^n$ diagonals apart from each other are also equal. Finally, entries on two diagonals symmetric within a palindrome are also equal.
\end{defi}

We prove our results on the limiting behavior (averaged over the ensemble) via Markov's Method of Moments (see for example \cite{Bi,Ta}) by showing that the average moments over the ensemble converge to the moments of a nice distribution. This, plus some control over the variance and the rate of convergence (done through a counting argument and an appeal to Chebyshev's inequality and the Borel-Cantelli lemma) suffice to prove various types of convergence of the limiting rescaled spectral measure to a fixed distribution. These convergence arguments are standard; see for example \cite{HM}.

\subsection{Results}

We fix a $p \in [1/2, 1]$ and study ensembles of signed structured matrices formed by multiplying the $(i,j)^{\text{th}}$ and $(j,i)^{\text{th}}$ entries of a matrix in our structured ensemble by a randomly chosen $\epsilon_{ij} \in \{1, -1\}$, with ${\rm Prob}(\epsilon_{ij} = 1)  = p$. As we vary $p$, we continuously interpolate between highly structured (when $p=1$) and less structured (when $p=1/2$) ensembles. As described above, our weighting is equivalent to taking the Hadamard matrix product of our original matrix and a real-symmetric sign matrix $(\epsilon_{ij})$. See \cite{GKMN} for results on Hadamard products of weight matrices and the adjacency matrices associated to $d$-regular graphs.

Unfortunately, in general it is very hard to obtain closed-form expressions for the limiting rescaled spectral measures (exceptions are the Gaussian behavior in singly palindromic Toeplitz and related behavior in block circulant ensembles \cite{MMS,KKMSX}, and Kesten's measure for $d$-regular graphs \cite{McK}); however, we are still able to prove many results about the moments of our signed, structured ensembles. For example, consider the Toeplitz ensembles. Using the expansion from the Eigenvalue Trace Lemma, a degree of freedom argument shows that the elements in the trace expansions must be matched in pairs; the difficulty is figuring out the contribution of each (which greatly depends on the structure of the matrix). The odd moments trivially vanish, and for even moments, the only contribution in the limit comes from when the indices are matched in pairs with opposite orientation. We show that we may view these terms as pairings of $2k$ vertices, $\left(i_{1}, i_{2}\right), \left(i_{2}, i_{3}\right), \ldots, \left(i_{2k}, i_{1}\right)$, on a circle.

We concentrate below on Toeplitz and related ensembles both for ease of presentation and because we can obtain more closed form results in some of these cases than is possible in general (and these results are related to questions in knot theory, which we discuss below), though our techniques apply to more general structured ensembles and we say a few words about these. Our main result is to show that the depression of the contribution of each pairing $c$ in the unsigned case for Toeplitz and singly palindromic Toeplitz matrices depends only on $e\left(c\right)$, where $e\left(c\right)$ is the number of vertices in crossing pairs in the pairing (we define these terms in \S\ref{sec:determiningmoments}). This extends previous results. When $p =1/2$, we are reduced almost completely to the real symmetric case, which means the limiting rescaled spectral measure is the semi-circle distribution (allowing special dependencies between matrix elements); our result also implies that all crossing configurations contribute $0$, and all non-crossing configurations contribute $1$. This gives us a $2k^{\text{th}}$ moment equal to the $k^{\text{th}}$ Catalan number, which is both the number of non-crossing pairings of $2k$ objects and the $2k^{\text{th}}$ moment of the semi-circle density.\footnote{The normalized semi-circular density is $f_{\rm sc}(x) = \frac{1}{\pi} \sqrt{1 - \left(\frac{x}{2}\right)^2}$ if $|x| \le 2$ and 0 otherwise, and the even moments are the Catalan numbers.}  By contrast, when $p = 1$ we are reduced to the unsigned case, and indeed our theorem implies that each configuration contributes what it did in the unsigned case. In addition, any distribution that had unbounded or bounded supported before weighting still has unbounded or bounded, respectively, support after weighting.

Our main result is the following.


\begin{thm}\label{thm:existencelimspecmeas}  Consider any ensemble of $N\times N$ real-symmetric structured matrices, where the independent entries are drawn from a distribution $\mathfrak{p}$ with mean 0, variance 1 and finite higher moments. We assume the following about our random matrix ensemble.

\begin{enumerate} 

\item As $N\to\infty$ the associated rescaled empirical spectral measures converge to a measure, which we call the limiting rescaled spectral measure of the structured ensemble and denote by $\widetilde{\mu}$.
    
\item Each of the independent random variables occurs $o(N)$ times\footnote{Little-oh notation: $f(x) = o(g(x))$ if $\lim_{x\to\infty} f(x)/g(x) = 0$; in particular, this means $f(x)$ grows significantly more slowly than $g(x)$.} in each row of the matrices for this ensemble.

\end{enumerate}

Fix a $p \in [1/2, 1]$ and consider the Hadamard product of our ensemble and real symmetric signed matrices $(\epsilon_{ij})$ (so $\epsilon_{ij} = \epsilon_{ji}$), where the entries are independently chosen from $\{-1,1\}$ with ${\rm Prob}(\epsilon_{ij} = 1)  = p$. We call this new ensemble the signed, structured ensemble.

For $p = 1/2$, the limiting rescaled spectral measures for these signed, structured ensembles are the semi-circle. For all other $p$, the limiting signed rescaled spectral measure has bounded (resp. unbounded) support if the original ensemble's limiting rescaled spectral measure has bounded (resp. unbounded) support, and the convergence is almost surely if additionally the density $\mathfrak{p}$ is even.
\end{thm}

\begin{rek}\label{rek:blocks} It is imperative that each independent random variable occurs at most $o(N)$ times; if one occurred order $N$ times degenerate behavior could happen. This precludes some highly structured matrices ensembles, such as those of the form \tiny $\mattwo{\alpha \mathcal{O}_N}{\beta \mathcal{O}_N}{\beta \mathcal{O}_N}{\gamma \mathcal{O}_N}$ \normalsize (with $\mathcal{O}_N$ the $N/2 \times N/2$ matrix all of whose entries are 1), where the limiting rescaled spectral measure is essentially a delta spike at the origin. Another interesting ensemble is the ``right angle'' family, where $a_{ij} = b_{\min(i,j)}$ (so there are $N$ independent random variables): \be \left(
\begin{array}{cccc}
b_1     & b_1       & b_1 & \cdots  \\
b_1     & b_2       & b_2 & \cdots \\
b_1     & b_2       & b_3 & \cdots \\
\vdots  & \vdots    & \vdots & \ddots \\
 \end{array}
\right). \ee
Notice both of these families have rows with order $N$ copies of the same random variable, and they have different behavior in their limiting spectral measures.
\end{rek}

\begin{cor}\label{cor:toepcase}
Theorem \ref{thm:existencelimspecmeas} holds for real-symmetric Toeplitz and singly palindromic Toeplitz matrices. More is true; see Theorem \ref{thm:weightedcontributions} for an explicit, closed form expression for the depression of the moments of these ensembles as $p \to 1/2$.
\end{cor}

The controlling factor in the real-symmetric Toeplitz and singly palindromic Toeplitz cases (and in a limited manner the highly palindromic Toeplitz cases) lurking in Corollary \ref{cor:toepcase} is how many vertices are involved in a crossing; we make this precise in \S\ref{sec:determiningmoments}. This reduces our problem to one in combinatorics. Our problem turns out to be related to issues in knot theory as well, which provided additional motivation for and applications of this work; see for example \cite{CM,KT,Kl2,Kont,FN,Rio,Sto}). In the course of our investigations, we prove several interesting combinatorial results (many of the coefficients have been previously tabulated on the OEIS; see for example Remark \ref{eq:kl1coeff}), which we isolate below.

\begin{thm}\label{thm:momentformulas} Consider all $\left(2k-1\right)!!$ pairings of $2k$ vertices on a circle. Let ${\rm Cr}_{2k, 2m}$ denote the number of these pairings where exactly $2m$ vertices are involved in a crossing, and let $C_{k}$ denote the $k^{\rm th}$ Catalan number, $\frac{1}{k+1}\binom{2k}{k}$. For small values of $m$, we obtain the exact formulas for ${\rm Cr}_{2k, 2m}$ listed below; for large $k$ (and thus a large range of possible $m$) we prove the limiting behavior of the expected value and variance of the number of vertices involved in at least one crossing.
\begin{itemize}
\item For $m \le 10$ we have \bea
{\rm Cr}_{2k, 0} & \ = \ & C_{k} \nonumber \\
{\rm Cr}_{2k, 2} & \ = \ & 0 \nonumber \\
{\rm Cr}_{2k, 4} & \ = \ & \binom{2k}{k-2} \nonumber\\
{\rm Cr}_{2k, 6} & \ = \ & 4\binom{2k}{k-3} \nonumber\\
{\rm Cr}_{2k, 8} & \ = \ & 31\binom{2k}{k-4}+\sum_{d=1}^{k-4}\binom{2k}{k-4-d}\left(4+d\right) \nonumber\\
{\rm Cr}_{2k, 10} & \ = \ & 288\binom{2k}{k-5}+8\sum_{d=1}^{k-5}\binom{2k}{k-5-d}\left(5+d\right).
\eea
\item The expected number of vertices involved in a crossing is \be \frac{2k}{2k-1}\left(2k-2 - \frac{{\ }_2F_1(1,3/2,5/2-k;-1)}{2k-3} - (2k-1){\ }_2F_1(1,1/2+k,3/2;-1)\right), \ee which is
\be\label{eq:pcrossaymptotic} 2k-2 - \frac2{k} + O\left(\frac1{k^2}\right)\ee as $k \rightarrow \infty$; here $_2F_1$ is the hypergeometric function. Further, the variance of the number of vertices involved in a crossing converges to $4$.
\end{itemize}

\end{thm}


We review the basic framework and definitions used in studying the moments in \S\ref{sec:momentprelims}. In \S\ref{sec:determiningmoments} we determine formulas for the moments, and prove the first part of Theorem \ref{thm:momentformulas} in the Toeplitz case, completing the proof by determining the limiting behavior in \S\ref{sec:limiting behavior} and discussing the minor changes needed for the general case. All that remains to prove Theorem \ref{thm:existencelimspecmeas} is to handle the convergence issues; this analysis is standard, and is quickly reviewed in \S\ref{sec:limitingspectralmeasure}.


\section{Moment Preliminaries}\label{sec:momentprelims}

\noindent \textbf{\emph{Note:} For ease of exposition we consider (real symmetric) Toeplitz ensembles below, though minor modifications yield similar results for other real symmetric structured ensembles where the limiting rescaled spectral measure exists and each random variable occurs $o(N)$ times in each row of matrices in the ensemble. In particular, we take $(\mathfrak{c}, r)$ to be $(1, 1/2)$.}\\ \

We briefly summarize the needed expansions from previous work (see \cite{HM,JMP,KKMSX,MMS} for complete details). We use a standard method to compute the moments. For a fixed $N\times N$ matrix $A$ drawn from a Toeplitz ensemble, the $k$\textsuperscript{th} moment of its rescaled empirical spectral measure is
\be
M_{k,N}\left(A\right)  \ = \  \frac{1}{N^{\frac{k}{2}+1}}\sum_{1 \leq i_{1},\ldots, i_{k} \leq N} a_{i_{1}i_{2}}a_{i_{2}i_{3}}\cdots a_{i_{k}i_{1}},
\ee
which when applied to our signed Toeplitz and palindromic Toeplitz matrices (where the entries of the unsigned ensemble are constant along diagonals) gives that
\be
M_{k,N}\left(A\right) \ = \ \frac{1}{N^{\frac{k}{2}+1}}\sum_{1 \leq i_{1}, \ldots, i_{k} \leq N} \epsilon_{i_{1}i_{2}}b_{\left|i_{1}-i_{2}\right|}\epsilon_{i_{2}i_{3}}b_{\left|i_{2}-i_{3}\right|} \cdots\epsilon_{i_{k}i_{1}}b_{\left|i_{k}-i_{1}\right|}.
\ee
By linearity of expectation,
\be\label{eq:expectedmoment}
\E\left(M_{k,N}\left(A\right)\right) \ = \  \frac{1}{N^{\frac{k}{2}+1}}\sum_{1 \leq i_{1},\ldots, i_{k} \leq N} \E\left( \epsilon_{i_{1}i_{2}}b_{\left|i_{1}-i_{2}\right|}\epsilon_{i_{2}i_{3}}b_{\left|i_{2}-i_{3}\right|}\cdots\epsilon_{i_{k}i_{1}}b_{\left|i_{k}-i_{1}\right|}\right),
\ee and we set \be M_k \ = \ \lim_{N\to\infty} \E\left(M_{k,N}\left(A\right)\right). \ee
Of the $N^{k}$ terms in the above sum corresponding to the $N^{k}$ choices of $\left(i_{1}, \ldots, i_{k}\right)$ in the above sum, we can immediately see that some contribute zero in the limit as $N \rightarrow \infty$ by using the following lemmas.

\begin{lem}\label{lem:pairs}
Let $k$ be an integer and consider any Toeplitz ensemble. The only terms in \eqref{eq:expectedmoment} that can have a non-zero contribution in the limit as $N\to\infty$ to $M_k$ have each $b_{\alpha}$ in the product appearing exactly twice. Further, all such terms have a finite contribution.
\end{lem}

\begin{proof}
We first prove that any term that doesn't have every $b_{\alpha}$ appearing at least twice does not contribute. As the expected value of a product of independent variables is the product of the expected values, since each $b_\alpha$ is drawn from a distribution with mean zero, there is no contribution in this case. Thus each $b_\alpha$ occurs at least twice if the term is to contribute.

We now show that any term that has some $b_{\alpha}$ appearing more than twice cannot contribute in the limit. If each $b_{\alpha}$ appears exactly twice, then there are $k/2$ values of $b_{\alpha}$ to choose. Recall (see for example \cite{HM}) that for Toeplitz matrices, $b_{\left|i_{j} i_{j+1}\right|}$ is paired with $b_{\left|i_{k} - i_{k+1}\right|}$ if and only if
\be\label{eq:pmpairing}
i_{j} - i_{j+1} \ = \pm (i_{k} - i_{k+1}).
\ee

Once we have specified the $b$'s and one index $i_{l}$, there are at most two values for each remaining index. Thus there are $O\left(N^{\frac{k}{2}+1}\right)$ terms where the $b_{\alpha}$'s are matched in exactly pairs. By contrast, any term that has some $b_{\alpha}$ appearing more than twice has fewer than $\frac{k}{2}+1$ degrees of freedom, and thus does not contribute in the limit as we divide by $N^{k/2 + 1}$.

Finally, we show that the sum of the contributions from all terms arising from matching in pairs is $O_{k}\left(1\right)$. Suppose there are $r \leq k$ different $\epsilon_{\gamma}$'s and $s \leq k$ different $b_{\alpha}$'s in the product, say $\epsilon_{\gamma_{1}},\ldots,\epsilon_{\gamma_{r}}$ and $b_{\alpha_{1}},\ldots,b_{\alpha_{s}}$, with each $\epsilon_{\gamma_{j}}$ occurring $n_{j}$ times and each $b_{\alpha_{j}}$ occurring $m_{j}$ times.  Such a term contributes $\prod_{j = 1}^{r}\E\left(\epsilon_{\gamma_j}^{n_j}\right) \prod_{j = 1}^{s}\E\left(b_{\alpha_j}^{m_j}\right)$. Since the probability distributions of the $\epsilon$'s and $b$'s have finite moments, this contribution is thus $O_{k}\left(1\right)$, and thus the sum of all such contributions is finite in the limit.
\end{proof}

\begin{rek}
For singly palindromic Toeplitz and highly palindromic Toeplitz matrices, a similar result holds once we identify the appropriate $b_\alpha$. After correcting equations (2.7) and (2.8) of \cite{JMP} to fix an omission and to take $\mathcal{C}_1 \in \{(-\lfloor \frac{|i_l - i_{l+1}|}{N/2^n}\rfloor + k-1)\frac{N}{2^n} : k \in \{1, \ldots, 2^n\}\}$ and $\mathcal{C}_2 \in \{(\lfloor \frac{|i_l - i_{l+1}|}{N/2^n}\rfloor + k)\frac{N}{2^n} - 1 : k \in \{1, \ldots, 2^n\}\}$ into account, we have that $b_{\left|i_{j} i_{j+1}\right|}$ is paired with $b_{\left|i_{k} - i_{k+1}\right|}$ if and only if
\be
i_{j} - i_{j+1} \ = \pm (i_{k} - i_{k+1}) + \mathcal{C}_{r_{jk}}.
\ee
For singly palindromic Toeplitz matrices, it is easy to check that the only possible values are $\mathcal{C}_{r_{jk}}$ equals $\pm (N-1)$ or $0$. Moreover, it is not hard to see that the number of possible values for each $\mathcal{C}_{r_{jk}}$ depends on the moment $m$ being computed and on the level $n$ of palindromicity of the ensemble, but is independent of $N$, a fact which will be crucially important in later proofs.
\end{rek}

\begin{lem}\label{lem:oddmoments} For Toeplitz and (highly) palindromic Toeplitz ensembles, the odd moments of the limiting rescaled spectral measure vanish.
\end{lem}

\begin{proof}
For the Toeplitz ensemble, this follows directly from Lemma \ref{lem:pairs} (since the odd moments have an odd number of $b$'s, they cannot be matched exactly in pairs). For the singly palindromic and highly palindromic cases, some $b_l$ must appear an odd number of times. If it appears exactly once, it must vanish because the distribution is mean zero, while the number of terms where some $b_l$ appears three or more times is insignificant by a simple degree of freedom argument. (For a more detailed exposition, see \cite{JMP}.)
\end{proof}

Since the odd moments vanish, we concern ourselves in the rest of the paper with the limiting behavior of the even moments, $M_{2k}$. Further, in the moment expansion for the even moments, we only have to consider terms in which the $b_{\alpha}$'s are matched in exactly pairs. With the next lemma, we further reduce the number of terms we must consider by showing that only those terms where every pairing between the $b$'s is with a minus sign in \eqref{eq:pmpairing} contribute in the limit. The following proof is adapted from \cite{HM}.

\begin{lem}\label{lem:negativesonly} For all the Toeplitz ensembles, the only terms that contribute to $M_{2k}$, the $2k\textsuperscript{\rm th}$ moment of the limiting rescaled spectral measure, are terms where the $b$'s are matched in exactly pairs and have a minus sign in each of  the $k$ equations of the form \eqref{eq:pmpairing}.
\end{lem}

\begin{proof} We do the proof for the Toeplitz case, as the other cases are similar. For each term, there are $k$ corresponding equations of the form \eqref{eq:pmpairing}. We let $x_{1}, \ldots, x_{k}$ be the values of the $\left|i_{j}-i_{j+1}\right|$ in these equations, and let $\delta_{1}, \ldots, \delta_{k}$ be the choices of sign in these equations. We further let $\tilde{x}_{1} = i_{1}-i_{2}$, $\tilde{x_{2}} = i_{2} - i_{3}, \ldots, \tilde{x}_{2k} = i_{2k}-i_{1}$. We know the only contribution to $M_{2k}$ arises from terms where the $b$'s are matched in pairs. Thus given some $\tilde{x}_m$ there must be an $n = n(m)$ such that $\tilde{x}_m = \pm \tilde{x}_n$. Then each of the previous $k$ equations can be written as
\be
\tilde{x}_{m}\ =\ \delta_{j}\tilde{x}_{n}, \ \ \ \delta_j \in \{-1, 1\}.
\ee
By definition, there is some $\eta_{j} = \pm 1$ such that $\tilde{x}_{m} = \eta_{j}x_{j}$. Then $\tilde{x}_{n} = \delta_{j}\eta_{j}x_{j}$, so
\be
\tilde{x}_{1} + \tilde{x}_{2} + \cdots + \tilde{x}_{2k} \ = \ \sum_{j=1}^{k}\eta_{j}\left(1+\delta_{j}\right)x_{j}.
\ee
Finally, notice that
\be
\tilde{x}_{1} + \tilde{x}_{2} + \cdots + \tilde{x}_{2k} \ =  \ i_{1} - i_{2} + i_{2} - i_{3} + \cdots + i_{2k} - i_{1} \ = \ 0.
\ee
Thus
\be\label{eq:signslineardep}
\sum_{j=1}^{k}\eta_{j}\left(1+\delta_{j}\right)x_{j} \ = \ 0.
\ee

If any $\delta_{j} = 1$, then \eqref{eq:signslineardep} gives us a linear dependence between the $x_{j}$. Recall from the proof of Lemma \ref{lem:pairs} that we require all $x_{j}$ to be independently chosen for a pairing to contribute; otherwise, there are fewer than $k+1$ degrees of freedom. Thus, the only terms that contribute have each $\delta_{j} = -1$.

From \cite{JMP}, the analogous result holds for the singly palindromic and highly-palindromic Toeplitz ensembles, i.e.,
\be
i_{j} - i_{j+1} \ = - (i_{k} - i_{k+1}) \pm \mathcal{C}_{r_{jk}}.
\ee
\end{proof}

The above results motivate the following definition.

\begin{defi}[Pairing] A \textbf{pairing} is a matching of the vertices $i_{1}, i_{2},\dots, i_{2k}$ such that the vertices are matched exactly in pairs, and with a negative sign in \eqref{eq:pmpairing}. There are $\left(2k-1\right)!!$ pairings of the $2k$ vertices. As argued above in the proof of Lemma \ref{lem:pairs}, these pairings correspond to $O\left(N^{k + 1}\right)$ terms in the sum in \eqref{eq:expectedmoment} for the $2k^{\rm {th}}$ moment.  \end{defi}

As suggested above, we find that a good way to investigate the contribution of each potentially contributing term, i.e., each choice or tuple of $\left(i_{1}, \ldots, i_{2k}\right)$, is to associate each term with a pairing of $2k$ vertices on a circle, where the vertices are $\left|i_{1}-i_{2}\right|, \left|i_{2}-i_{3}\right|, \ldots, \left|i_{2k}-i_{1}\right|$. Because what matters are not the values of the $\left|i_{j}-i_{j+1}\right|$'s, but rather the pattern of how they are matched, any terms associated with the same pairing of the $2k$ vertices will have the same contribution. Thus, pairings that are the same up to a rotation of the vertices contribute the same since it is not the values of $i_{j}$ that matter but rather the distance between each vertex and its matching and the indices of the other pairs. Therefore, to further simplify the moment analysis, we make the following definition.

\begin{center}
\begin{figure}\label{sixthMomentConfig}
\includegraphics[height=30mm]{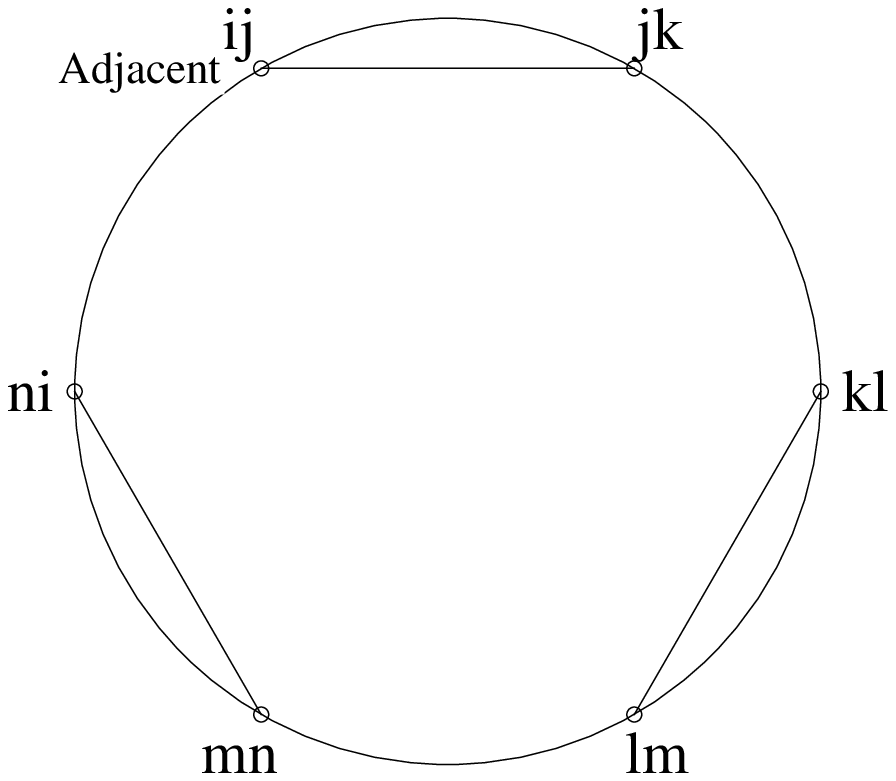}\ \includegraphics[height=30mm]{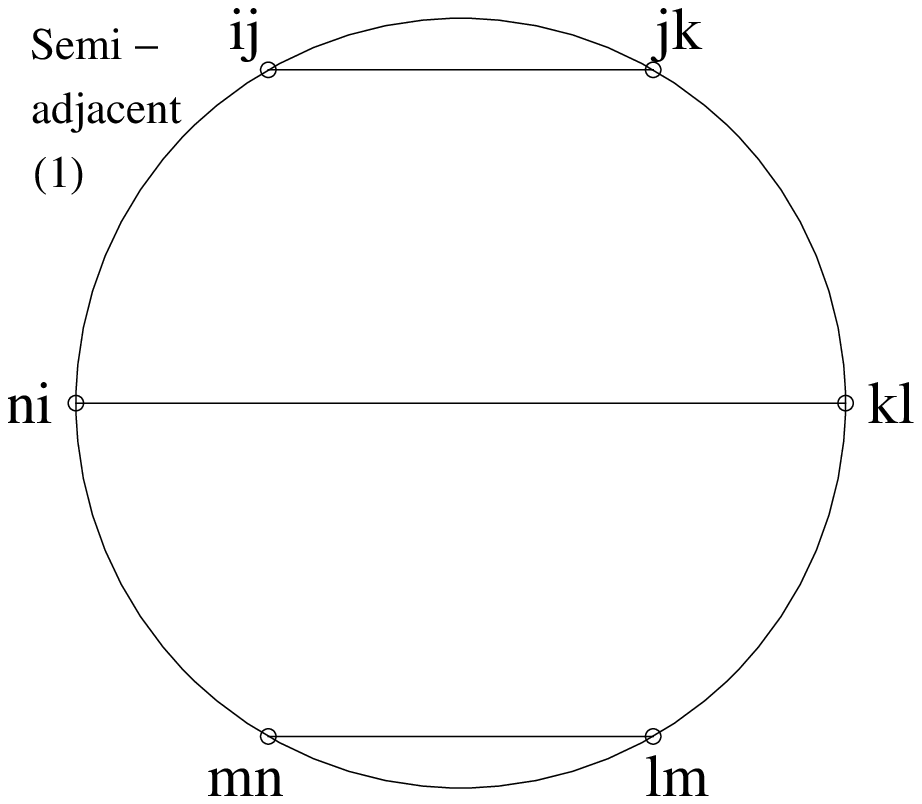}\  \includegraphics[height=30mm]{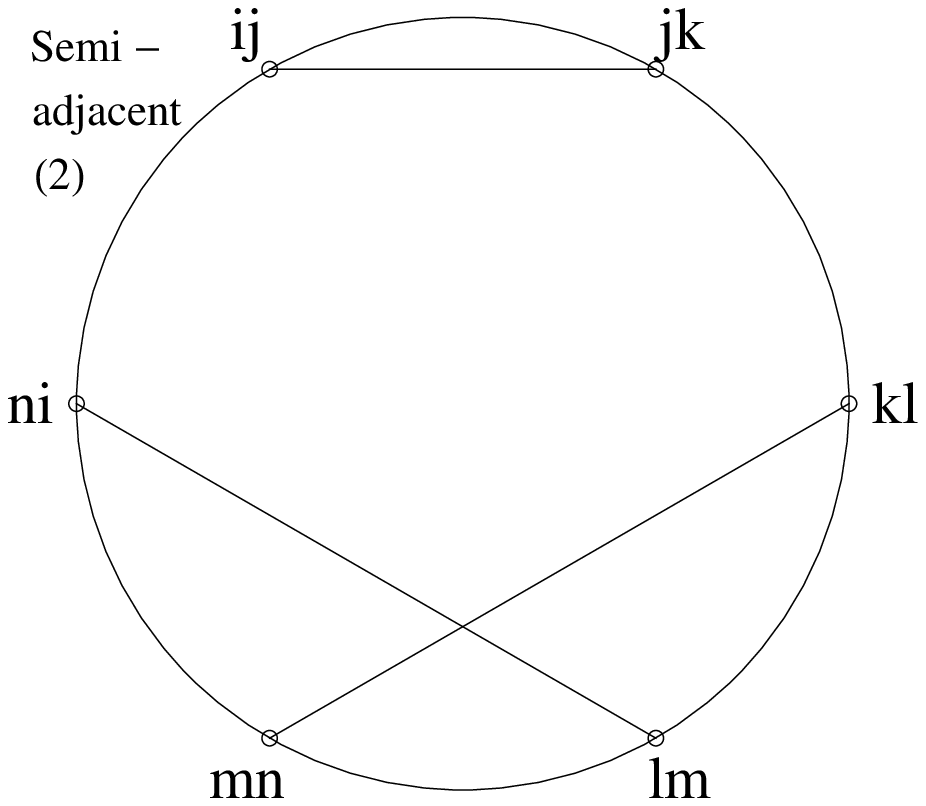}\ \includegraphics[height=30mm]{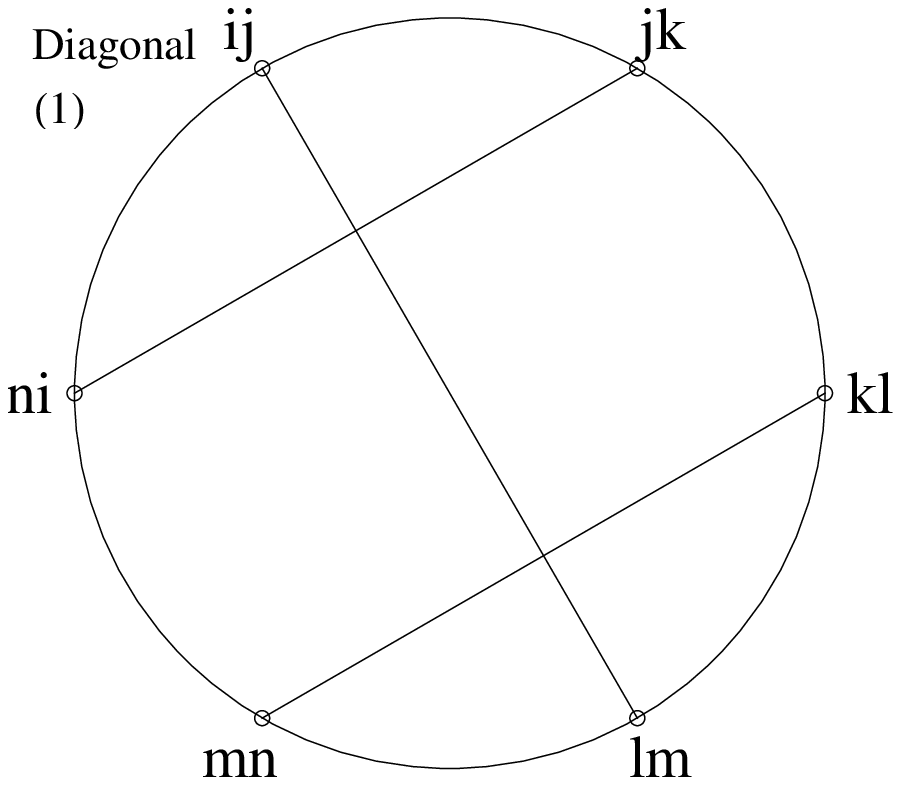}\  \includegraphics[height=30mm]{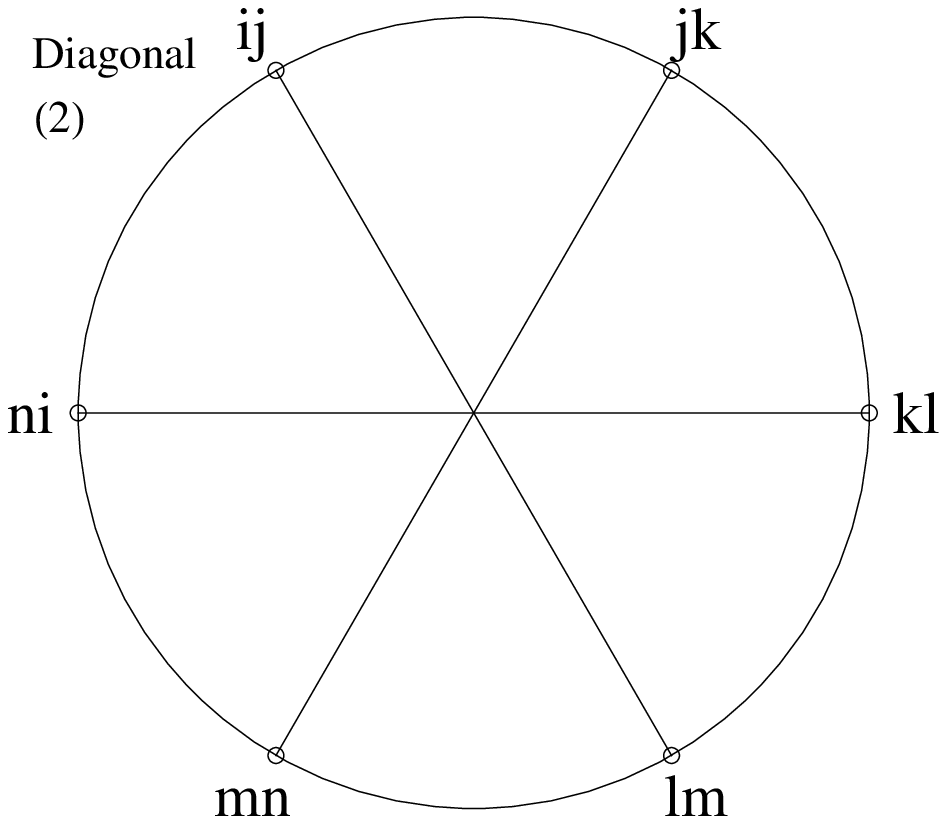}\
\caption{\label{fig:6moments} The five distinct configurations for the 6\textsuperscript{th} moment where vertices are matched exactly in pairs. The multiplicity under rotation of the five patterns are 2, 3, 6, 3 and 1 (for example, rotating the first pattern twice returns it to its initial configuration, while the third requires six rotations). The nomenclature is from \cite{KKMSX}, and is not relevant to our purposes here.}
\end{figure}
\end{center}

\begin{defi}[Configuration] Two pairings $\left\{\left(i_{a_{1}}, i_{a_{2}}\right), \left(i_{a_{3}}, i_{a_{4}}\right), \ldots, \left(i_{a_{2k-1}}, i_{a_{2k}}\right)\right\}$ and $\left\{\left(i_{b_{1}}, i_{b_{2}}\right)\right.$, $\left(i_{b_{3}}, i_{b_{4}}\right)$, $\ldots$, $\left.\left(i_{b_{2k-1}}, i_{b_{2k}}\right)\right\}$ are said to be in the same configuration if they are equivalent up to a relabeling by rotating the vertices; i.e., there is some constant $l$ such that $b_{j} = a_{j} + l \mod{2k}$. \end{defi}

For example, we display the five distinct configurations needed for the sixth moment in Figure \ref{fig:6moments}. The problem of determining the moments is thus reduced to determining for each configuration both the contribution of a pairing belonging to that configuration to the sum in \eqref{eq:expectedmoment} and the number of pairings belonging to that configuration.


\section{Determining the Moments}\label{sec:determiningmoments}

By Lemma \ref{lem:pairs}, for the rest of the paper we may assume the vertices are matched in exactly pairs. We distinguish between three types of vertices in these pairings.

\begin{figure}
\includegraphics[height=40mm]{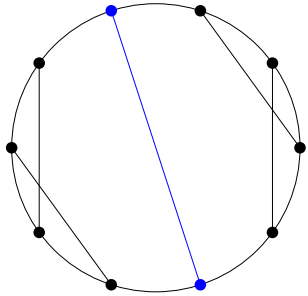}
\caption{A pairing of $10$ vertices with $8$ crossing vertices (in two symmetric sets of 4 vertices), and $2$ dividing vertices (connected by a main diagonal).}\label{fig:vertextypes}
\end{figure}

\begin{defi}[Crossing, non-crossing]
We say that a pair of vertices $\left(a, b\right)$, $a < b$, is in a \textbf{crossing} if there exists a pair of vertices $\left(x, y\right)$ such that the order of the four vertices, as we travel clockwise around the circle, is either $a, x, b, y$ or $x, a, y, b$. A pair $\left(a, b\right)$ is \textbf{non-crossing} if for every pair $\left(x, y\right)$, $x$ is between $a$ and $b$ (as we travel clockwise around the circle from $a$ to $b$) if and only if $y$ is.
\end{defi}

Pictorially, a pair is crossing if the line contained in the circle connecting its two vertices crosses another line connecting two other vertices. In Figure \ref{fig:6moments}, the first two configurations have no crossing vertices, the third has four, while all vertices are crossing for the fourth and fifth. Note the number of crossing vertices is always even and never two.

\begin{defi}
We say that a non-crossing pair of vertices $\left(a, b\right)$ (with $a < b$) is \textbf{dividing} if the following two conditions hold:

\begin{enumerate}

\item There exist two pairs of crossing vertices, $\left(x, y\right)$ and $\left(w, z\right)$, such that as we travel around the circle from $a$ to $b$ we have $x, y, w$ and $z$ are between $a$ and $b$.

\item There exist two pairs of crossing vertices,  $\left(p, q\right)$ and $\left(r, s\right)$, such that as we travel around the circle from $b$ to $a$ we have $p, q, r$ and $s$ are between $b$ and $a$.

\end{enumerate}

All other pairs are called \textbf{non-crossing non-dividing} pairs. \end{defi}

Pictorially, a pair is dividing if it ``divides'' the circle into two regions of pairs (no pair can cross a dividing edge since it must be non-crossing), where each region contains at least one crossing pair; see Figure \ref{fig:vertextypes} for an illustration. From the definition, we see that at least $10$ vertices are needed for a ``dividing'' pair to exist, and thus it is possible that new behaviors or complications arise in studying the higher moments (a similar situation arises in weighted $d$-regular graphs, where there is a marked change in behavior at the eighth moment; see \cite{GKMN} for details).

Note that all pairings belonging to a given configuration have the same number of crossing pairs and the same number of dividing pairs.

We show in this section that the contribution of each pairing in the unsigned case is weighted by a factor depending on the number of crossing pairs in that pairing. We then prove some combinatorial formulas that allow us to obtain closed form expressions for the number of pairings with $m$ vertices crossing for small $k$. As the combinatorics becomes prohibitively difficult for large $k$, we determine the limiting behavior in \S\ref{sec:limiting behavior}.

\subsection{Weighted Contributions}

The following theorem is central to our determination of the moments. It reduces the calculations to two parts. First, we need to know the contribution of a pairing in the non-signed case (equivalently, when $p=1$). While this is known precisely for the singly palindromic Toeplitz case, where each pairing contributes 1, in the Toeplitz case we only have upper and lower bounds on the contribution of all pairing. Second, we need to determine the number of vertices involved in crossing pairs, which we do in part in \S\ref{sec:countingcrossingconfigs}.

\begin{rek}
For ease of exposition, we prove the following lemmas in the Toeplitz case, and comment on the proofs (or barriers to proof) in the singly palindromic and highly palindromic cases. For the palindromic case, by (2.7) and (2.8) of \cite{JMP}, there should be some $\mathcal{C}_1$ and $\mathcal{C}_2$ terms added into equation (\ref{eq:pmpairing}) as well as parts of the proof for Lemma \ref{lem:negativesonly}; however, some minor changes to the proofs show that these lemmas still hold in the palindromic Toeplitz case.
\end{rek}

\begin{thm}\label{thm:weightedcontributions}
For each choice of a pairing $c$ of the vertices $\left(i_{1}, \ldots, i_{2k}\right)$, let $x(c)$ denote the contribution of this tuple in the unsigned case. Then, for the Toeplitz and singly palindromic Toeplitz ensembles, the contribution in the signed case is $x(c) (2p-1)^{e(c)}$, where $e(c)$ represents the number of vertices in crossing pairs in the configuration corresponding to $c$.
\end{thm}

Recall that the contribution from any choice of $\left(i_{1}, \ldots, i_{2k}\right)$ is
\bea
\E(\epsilon_{i_1 i_2} b_{|i_1 - i_2|} \epsilon_{i_2 i_3} b_{|i_2 - i_3|} \cdots \epsilon_{i_{2k} i_1} b_{|i_{2k} - i_1}|) & \ = \ &
\E(\epsilon_{i_1 i_2} \epsilon_{i_2 i_3} \cdots \epsilon_{i_{2k} i_1}) \E(b_{|i_1 - i_2|} \cdots b_{|i_k - i_1|}) \nonumber\\ & \ = \ & \E(\epsilon_{i_1 i_2} \epsilon_{i_2 i_3} \cdots \epsilon_{i_{2k} i_1}) x(c).
\eea

Thus, we want to show that $\E(\epsilon_{i_1 i_2} \epsilon_{i_2 i_3} \cdots \epsilon_{i_{2k} i_1}) = (2p-1)^{e(c)}$. We do this by showing that for each pair $\left(i_{j}, i_{j+1}\right), \left(i_{k}, i_{k+1}\right)$ where $b_{\left|i_{j}-i_{j+1}\right|} = b_{\left|i_{k}-i_{k+1}\right|}$,
\begin{equation}
\twocase{\E\left(\epsilon_{i_{j}i_{j+1}}\epsilon_{i_{k}i_{k+1}}\right) \ = \ }{\left(2p-1\right)^{2}}{if $\left(i_{j}, i_{j+1}\right), \left(i_{k}, i_{k+1}\right)$ are a crossing pair}{1}{otherwise.}
\end{equation}

Notice that
\be
\E\left(\epsilon_{\alpha}\right)\ =\ 1 \cdot p + (-1) \cdot (1-p) \ = \ 2p-1, \ \ \ \
\E\left(\epsilon_{\alpha}^{2}\right)\ =\ 1.
\ee
Therefore, if $m$ epsilons are chosen independently, the expected value of their product is $(2p-1)^m$.

Before stating and proving some lemmas needed in the proof of Theorem \ref{thm:weightedcontributions}, we introduce a convenient notation.

\begin{defi}[Vertex ordering]
Fix an integer $2k$ and consider the circle with $2k$ vertices spaced uniformly, labeled 1, 2, $\dots$, $2k$. If $a, b$ and $x$ are three of these vertices, by $a < x < y$ we mean that we pass through vertex $x$ as we travel clockwise about the circle from vertex $a$ to vertex $b$. \end{defi}

\begin{lem}\label{lem:LargeContribution}
For the Toeplitz and singly palindromic Toeplitz ensembles,
$\E(\epsilon_{i_1 i_2} \epsilon_{i_2 i_3} \cdots \epsilon_{i_{2k} i_1}) \geq (2p-1)^{e(c)}$.
\end{lem}

\begin{proof}
To prove $\E(\epsilon_{i_1 i_2} \epsilon_{i_2 i_3} \cdots \epsilon_{i_{2k} i_1}) \geq (2p-1)^{e(c)}$, we show that pairs not in a crossing contribute $1$. Consider a non-crossing pair $\left(i_r, i_{r+1}\right), \left(i_p, i_{p+1}\right)$ (corresponding to vertices $r$ and $p$ on the circle with $2k$ labeled vertices), with $r < p$. For each $\left(i_q, i_{q+1}\right)$ paired with $\left(i_{q'}, i_{q'+1}\right)$, we have $r < q < p$ if and only if $r < q' < p$. Recall from \eqref{eq:pmpairing} and Lemma \ref{lem:negativesonly} that in the Toeplitz case,
\be\label{form1}
i_q - i_{q+1} \ = \ -(i_{q'} - i_{q'+1}),
\ee
while in the singly palindromic Toeplitz case,
\be\label{form2}
i_q - i_{q+1} \ = \ -(i_{q'} - i_{q'+1}) + Q(q,q'), \ \ \ {\rm where}\ Q(q,q') \in \{-(N -1), 0, N-1\}.
\ee
Thus
\be
\sum_{k=r}^p (i_{k} - i_{k+1}) \ = \ t(N-1)
\ee for some integer $t$ because each difference in the sum is paired with its additive inverse, which is also in the sum.
As\be
\sum_{k=r}^p (i_{k} - i_{k+1}) \ = \ (i_r - i_{r+1}) + (i_{r+1} - i_{r+1}) + \cdots + (i_p - i_{p+1})\ =\ i_{r} - i_{p+1},
\ee we must have $i_r = i_{p+1} \pm t(N-1)$. It is clearly impossible to have $|t| > 1$, and if $t = \pm 1$, this forces $\{i_r, i_{p+1}\} = \{1,N\}$; thus $t=0$. Since this situation uses up a degree of freedom, this implies that $i_r = i_{p+1}$. By a similar argument applied to the sum
\be
\sum_{k=p}^r (i_{k} - i_{k+1})
\ee
(taking indices cyclically), $i_{r+1} = i_p$. Therefore $\epsilon_{i_r i_{r+1}} = \epsilon_{i_p i_{p+1}}$, and hence $\E(\epsilon_{i_r i_{r+1}} \epsilon_{i_p i_{p+1}}) = 1$.
\end{proof}

\begin{lem}\label{lem:SmallContribution}
For the Toeplitz, singly palindromic Toeplitz, and highly palindromic Toeplitz ensembles,
$\E(\epsilon_{i_1 i_2} \epsilon_{i_2 i_3} \cdots \epsilon_{i_{2k} i_1}) \leq (2p-1)^{e(c)}$.
\end{lem}

\begin{proof}
We show $\E(\epsilon_{i_1 i_2} \epsilon_{i_2 i_3} \cdots \epsilon_{i_{2k} i_1}) \leq (2p-1)^{e(c)}$ by showing that if $\epsilon_{i_a i_{a+1}} = \epsilon_{i_b i_{b+1}}$, $a <b$, then $\left(i_{a}, i_{a+1}\right), \left(i_{b}, i_{b+1}\right)$ are non-crossing. This suffices to prove the result since  the only dependency between the $\epsilon$'s arises from the requirement that the matrix is real symmetric. Thus, we have a dependency between $\epsilon_{i_s i_{s+1}}$ and $\epsilon_{i_p i_{p+1}}$ if and only if we know they are equal. In showing that a dependency between $\epsilon$'s implies the corresponding vertex pair must be non-crossing, we show that crossing pairs imply independent $\epsilon$'s and thus contribute $\left(2p-1\right)^{2}$.

If $\epsilon_{i_{a}i_{a+1}} = \epsilon_{i_{b}i_{b+1}}$ then it must be true that the unordered sets $\left\{i_{a}, i_{a+1}\right\}$ and $\left\{i_{b}, i_{b+1}\right\}$ are equal. This implies that $|i_a - i_{a+1}| = |i_b-i_{b+1}|$, so $\left(i_{a}, i_{a+1}\right), \left(i_{b}, i_{b+1}\right)$ must be paired on the circle. Since the only contributing terms are when they are paired in opposite orientation, we then know that  $i_a = i_{b+1}$, so
\be\label{form3}
\sum_{k=a}^b (i_{k} - i_{k+1}) \ = \ i_{a} - i_{b+1}\ =\ \sum_k \pm \mathcal{C}_{r_k}.
\ee
We can rewrite this sum as
\be\label{form4}
\sum_{k=b}^d \delta_k |i_{k} - i_{k+1}| \ = \ \sum_k \pm \mathcal{C}_{r_k},
\ee
where $\delta_k$ is $\pm 1$ if the vertex $k$ is paired with is less than $a$ or greater than $b$, and $0$ if and only if the vertex $k$ is paired with is between $a$ and $b$. However, since the number of possible values for $\sum_{k} \pm \mathcal{C}_{r_k}$ is independent of $N$, a linear dependence among the differences is impossible, as we need to have $N^{k+1}$ degrees of freedom for each configuration (see the proof of Lemma \ref{lem:pairs}). So each $\delta_k = 0$, and each vertex between vertices $a$ and $b$ is paired with something else between $a$ and $b$. Thus, no edges cross the edge between vertices $a$ and $b$.
\end{proof}

\begin{proof}[Proof of Theorem \ref{thm:weightedcontributions}] For Toeplitz and singly palindromic Toeplitz matrices, we have shown that an epsilon is unmatched if and only if its edge is in a crossing. Thus, an epsilon is not paired if and only if its edge is not in a crossing. Therefore the contribution is weighted by $\E(\epsilon_{i_1 i_2} \epsilon_{i_2 i_3} \cdots \epsilon_{i_{2k} i_1})$, which by Lemmas \ref{lem:LargeContribution} and \ref{lem:SmallContribution} is $(2p-1)^{e(c)}$, completing the proof.
\end{proof}

\begin{rek}
In the doubly palindromic Toeplitz case, Lemma \ref{lem:LargeContribution} does not hold for the sixth moment, as we shall see in Lemma \ref{SixthMomentDifferent}. In particular, this means the determination of the limiting rescaled spectral measures for general signed ensembles and general $p$ is harder.
\end{rek}

\begin{lem}\label{lem:E(c)=0}
For the Toeplitz, singly palindromic Toeplitz, and highly palindromic Toeplitz ensembles, if the contribution from a non-crossing configuration was $x$ before the weighting, it is at most $(2p-1)^4(x-1) + 1$ after applying the weighting.
\end{lem}

\begin{proof}
In the Toeplitz and singly palindromic Toeplitz cases, $x = 1$ and the claim is trivial. In the highly palindromic case, we note that there is a contribution of $1$ from the terms which also contribute in the real symmetric case. The remaining terms contain at least 2 pairs of vertices which are not matched in the real symmetric case, since one mismatched pair (relative to the real symmetric ensemble) implies a second mismatched pair, since $\sum_{k=1}^{2m} (i_k - i_{k+1}) = 0$. Hence, for these terms, $\E(\epsilon_{i_1 i_2} \epsilon_{i_2 i_3} \cdots \epsilon_{i_{2k} i_1})  \leq (2p-1)^{4}$, which completes the proof.
\end{proof}

\begin{rek}\label{rek:keyforcontributions}
A slightly modified version of this proof shows that for other real symmetric ensembles, if the contribution from a non-crossing configuration was $x$ before the weighting, it is at most $(2p-1)^2(x-1) + 1$ after applying the weighting. Similarly, for crossing configurations, if the contribution was $x$ before the weighting, it is at most $(2p-1)^2 x$ after applying the weighting.
\end{rek}

\begin{lem}\label{SixthMomentDifferent}
For the sixth moment of signed doubly palindromic Toeplitz ensembles, the contribution from a configuration is not determined uniquely by the number of crossings.
\end{lem}

\begin{proof}
We prove that the adjacent configuration and the non-adjacent non-crossing configuration (the upper-left and upper-middle configurations in Figure 1, respectively) have different contributions to the sixth moment.

The main idea is that in the `adjacent configuration', every contributing term has either all three pairings of the form $a_{ij}a_{ji}$, or exactly one pairing of this form. Since we know that the contribution when all three pairings are of this form is $1$, the contribution when there is exactly one pairing of this form is $(x-1)$. In this situation, the contribution to the moment is weighted by $(2p - 1)^4$, giving a total of $(2p - 1)^4 (x-1) + 1$.

Specifically, we have that \be
i_t - i_{t+1}\ =\ -(i_{t+1} - i_{t+2}) \pm \mathcal{C}_{r_{t,t+1}},
\ee
where $\mathcal{C}_{r_{t,t+1}} = N/2$ or $N/2 - 1$ or $0$. ($N$ and $N-1$ are ruled out because we would lose a degree of freedom by forcing one value to be $1$ and the other to be $N$.) Moreover, $\mathcal{C}_{r_{t,t+1}} = 0$ if and only if $\epsilon_{i_t i_{t + 1}} = \epsilon_{i_{t + 1} i_{t + 2}}$. Now, if we choose three values from $\{0, \pm N/2, \pm N/2 - 1\}$ that add up to $0$, we must choose either one or three of the values to be $0$. The cases where all three are $0$ contribute fully while the case where two are non-zero is depressed by $(2p - 1)^4$, so that contribution to the moment in the signed ensemble is exactly $(2p - 1)^4 (x-1) + 1$.

In the other non-crossing configuration, the moment is at most $(2p - 1)^4 (x-1) + 1$ by the proof of Lemma \ref{lem:E(c)=0}. Hence, to show the moment is smaller than this, it will suffice to find a contributing group of terms whose moment is reduced by more than $(2p - 1)^4$. As one example, we can take the vertices to be $a_{i,j}, a_{j,i + N/2}, a_{i + N/2,k + N/2}, a_{k + N/2,l}, a_{l,k}, a_{k,i}$, where $i, k < N/2$. While there is an additional inequality between $i$ and $j$ and between $k$ and $l$, this does not remove a degree of freedom since there are still order $N$ possible values. Hence, some portion of the $(x-1)$ contribution is reduced by a factor of $(2p - 1)^6 < (2p - 1)^4$. Since the remaining portion of the contribution is reduced to at most $(2p - 1)^4$ times its original value, the contribution to the 6\textsuperscript{th} moment of the non-adjacent non-crossing configuration in the signed doubly palindromic case is strictly less than $(2p - 1)^4(x - 1) + 1$, and is therefore not equal to the contribution from the adjacent non-crossing configuration.
\end{proof}



\subsection{Counting Crossing Configurations}\label{sec:countingcrossingconfigs}

Theorem \ref{thm:weightedcontributions} reduces the determination of the moments to counting the number of pairings with a given contribution $x(c)$, and then weighting those by $(2p-1)^{e(c)}$, where $e(c)$ is the number of vertices involved in crossings in the configuration. As remarked above, in the singly palindromic Toeplitz case each $x(c) = 1$, while in the general Toeplitz case we only have bounds on the $x(c)$'s, and thus must leave these as parameters in the final answer (though any specific $x(c)$ may be computed by brute force, we do not have a closed form expression in general).

In this section we turn to computing the $e(c)$'s for various configurations. As previously mentioned, these and similar numbers have also been studied in knot theory where these chord diagrams are used in the study of Vassiliev invariants (see \cite{KT,Kont,FN,Rio,Sto}). While we cannot determine exact formulas in general, we are able to solve many special cases, which we now describe.

\begin{defi}[${\rm Cr}_{2k, 2m}$] Let ${\rm Cr}_{2k, 2m}$ denote the number of pairings involving $2k$ vertices where exactly $2m$ vertices are involved in a crossing.
\end{defi}

Let $C_k = \frac1{k+1} \ncr{2k}{k}$ denote the $k$\textsuperscript{th} Catalan number (see \cite{AGZ} for statements and proofs of their needed properties). One of its many definitions is as the number of ways to match $2k$ objects on a circle in pairs without any crossings; this interpretation is the reason why Wigner's Semi-Circle Law holds. Thus, we immediately deduce the following.

\begin{lem}\label{lem:catalanlem} We have $\ccr{2k}{0} = C_k$. \end{lem}

We use this result to prove the following theorem, which is instrumental in the counting we need to do.

\begin{thm}\label{thm:2kCk-v} Consider $2k$ vertices on the circle, with a partial pairing on a subset of $2v$ vertices. The number of ways to place the remaining $2k-2v$ vertices in non-crossing, non-dividing pairs is $\binom{2k}{k-v}$.
\end{thm}
\begin{proof} Let $\mathcal{W}$ denote the desired quantity. Notice that each of the remaining $2k-2v$ vertices must be placed between two of the $2v$ already paired vertices on the circle. These $2v$ vertices have created $2v$ regions. A necessary and sufficient condition for these $2k-2v$ vertices to be in non-crossing, non-dividing pairs is that the vertices in each of these $2v$ regions pair only with other vertices in that region in a non-crossing configuration.

Thus, if there are $2s$ vertices in one of these regions, by Lemma \ref{lem:catalanlem} the number of valid ways they can pair is $C_{2s}$. As the number of valid matchings in each region depends only on the number of vertices in that region and not on the matchings in the other regions, we obtain a factor of $C_{2s_{1}}C_{2s_{2}}\cdots C_{2s_{2v}}$, where $2s_{1}+2s_{2}+\cdots+2s_{2v} = 2k-2v$.

We need only determine how many pairings this factor corresponds to. First we notice that by specifying one index and $\left(s_{1}, s_{2}, \ldots, s_{2v}\right)$, we have completely specified a pairing of the $2k$ vertices. However, as we are pairing on a circle, this specification does not uniquely determine a pairing since the labelling of $\left(s_{1}, s_{2}, \ldots, s_{2v}\right)$ is arbitrary. Each pairing can in fact be written as any of the $2v$ circular permutations of some choice of $\left(s_{1}, s_{2}, \ldots, s_{2v}\right)$ and one index. Thus the quantity we are interested in is
\be\label{eq:catconv}\mathcal{W} \ = \
\frac{2k}{2v}\sum_{2s_{1}+2s_{2}+\cdots+2s_{2v} = 2k-2v}C_{s_{1}}C_{s_{2}}\cdots C_{s_{2v}}.
\ee

To evaluate this expression, we use the $k$-fold self-convolution identity of Catalan numbers \cite{Fo,Reg}, which states
\be
\sum_{i_{1}+\cdots+i_{r} = n} C_{i_{r}-1}\cdots C_{i_{r}-1} \ = \ \frac{r}{2n-r}\binom{2n-r}{n}.
\ee
Setting $i_{j} = s_{j} + 1$, $r = 2v$ and $n = k+v$, we obtain
\be
\sum_{s_{1}+s_{2}+\cdots+s_{2v} + 2v = k+v} C_{s_{1}}C_{s_{2}}\cdots C_{s_{2v}} \ = \ \frac{2v}{2k}\binom{2k}{k+v}.
\ee
We may rewrite this as
\be
\frac{2k}{2v}\sum_{2s_{1}+2s_{2}+\cdots+2s_{2v} = 2k-2v} C_{s_{1}}C_{s_{2}}\cdots C_{s_{2v}} \ = \ \binom{2k}{k-v},
\ee which completes the proof as the left hand side is just \eqref{eq:catconv}.
\end{proof}

Given Theorem \ref{thm:2kCk-v}, our ability to find formulas for ${\rm Cr}_{2k, 2m}$ rests on our ability to find the number of ways to pair $2v$ vertices where $2m$ vertices are crossing and $2v-2m$ vertices are dividing. We are able to do this for small values of $m$, but for large $m$, the combinatorics becomes very involved.

\begin{defi}[${\rm P}_{2k, 2m, i}$, partitions] Let ${\rm P}_{2k, 2m, i}$ represent the number of pairings of $2k$ vertices with $2m$ crossing vertices in $i$ partitions. We define a partition to be a set of crossing vertices separated from all other sets of crossing vertices by at least one dividing edge.
\end{defi}

It takes a minimum of $4$ vertices to form a partition, so the maximum number of partitions possible is $\lfloor 2m/4 \rfloor$. Our method of counting involves writing
\begin{equation}\label{eq:sumpartitions}
{\rm Cr}_{2k, 2m} \ = \ \sum_{i=1}^{\lfloor 2m/4\rfloor}{\rm P}_{2k, 2m, i}.
\end{equation}

Our first combinatorial result is the following.

\begin{lem}\label{lem:p1} We have
\be
{\rm P}_{2k, 2m, 1} \ = \ {\rm Cr}_{2m, 2m} \binom{2k}{k-m}.
\ee
\end{lem}
\begin{proof}
The proof follows immediately from Theorem \ref{thm:2kCk-v}. If there is only one partition, then there can be no dividing edges. Therefore, we simply multiply the number of ways we can choose $2k-2m$ non-crossing non-dividing pairs by the number of ways to then choose how the $2m$ crossing vertices are paired.
\end{proof}

Our next result is

\begin{lem}\label{lem:p2} We have
\begin{equation}
{\rm P}_{2k, 2m, 2} \ = \  \sum_{d=1}^{k-m}\binom{2k}{k-m-d}\left(m+d\right)\left(\sum_{0 < a < m} {\rm Cr}_{2a, 2a} {\rm Cr}_{2m-2a, 2m-2a}\right).
\end{equation}
\end{lem}

\begin{proof}
We let $d$ be the number of dividing edges. In order to have two partitions, at least one of the $k-m$ non-crossing edges must be a dividing edge. We thus sum over $d$ from $1$ to $k-m$. Given $d$, we know that we can pair and place the non-crossing non-dividing edges in $\binom{2k}{k-m-d}$ ways from Theorem \ref{thm:2kCk-v}. We then choose a way to pair the $2m$ crossing vertices into $2$ partitions, one with $2a$ vertices, the other with $2b$ vertices. If $a = b$, there are $m+d$ distinct spots where we may place the dividing edge. If $a \neq b$, there are $2m+2d$ spots. Since each choice of $a \neq b$ appears twice in the above sum, the result follows.
\end{proof}

Determining ${\rm P}_{2k, 2m, 3}$ requires the analysis of several more cases, and we were unable to find a nice way to generalize the results of Lemmas \ref{lem:p1} and \ref{lem:p2}. However, these two results do allow us to write down the following formulas.

\begin{lem}\label{lem:crossingformulas} We have
\bea\label{eq:Cr2k4}
{\rm Cr}_{2k, 4} & \ = \ & \binom{2k}{k-2} \nonumber\\
{\rm Cr}_{2k, 6} & \ = \ & 4\binom{2k}{k-3} \nonumber\\
{\rm Cr}_{2k, 8} & \ = \ & 31\binom{2k}{k-4}+\sum_{d=1}^{k-4}\binom{2k}{k-4-d}\left(4+d\right) \nonumber\\
{\rm Cr}_{2k, 10} & \ = \ & 288\binom{2k}{k-5}+8\sum_{d=1}^{k-5}\binom{2k}{k-5-d}\left(5+d\right).
\eea
\end{lem}

\begin{proof}
We recall that
\bea\label{eq:initcond}
{\rm Cr}_{2k, 0} & \ = \ & C_{k} \nonumber\\
{\rm Cr}_{2k, 2} & \ = \ & 0,
\eea
where the second equation follows from the fact that at least $4$ vertices are needed for a crossing. From \eqref{eq:sumpartitions} and \eqref{lem:p1} we find
\begin{equation}
{\rm Cr}_{2k, 4}\ =\ {\rm P}_{2k, 4, 1}\ =\ {\rm Cr}_{4, 4}\binom{2k}{k-2}.
\end{equation}

We can calculate ${\rm Cr}_{4, 4}$ by using \eqref{eq:initcond} and the fact that
\be\label{eq:totalcrossingconfigs}
\sum_{m = 0}^{k} {\rm Cr}_{2k, 2m}\ =\ \left(2k-1\right)!!.
\ee This follows because the number of ways to match $2k$ objects in pairs of 2 with order not mattering is $(2k-1)!!$, and thus the sum of all our matchings in pairs must equal this. Note that this number is also the $2k$\textsuperscript{th} moment of the standard normal; this is the reason the singly palindromic Toeplitz have a limiting rescaled spectral measure that is normal, as each contribution contributes fully. We thus find
\begin{equation}
{\rm Cr}_{4, 4}\ = \ \left(2\cdot 2-1\right)!! - {\rm Cr}_{4, 2} - {\rm Cr}_{4, 0}\ =\ 3 - 2\ =\ 1.
\end{equation}
This completes the proof of the first formula: ${\rm Cr}_{2k, 4} = \binom{2k}{k-2}$.

The other coefficients are calculated in a similar recursive fashion -- essentially, once we have values for ${\rm Cr}_{2k, 2l}$ for $l = 0, 1, 2, \ldots, m-1$, we can find ${\rm Cr}_{2m, 2m}$ by using \eqref{eq:totalcrossingconfigs}, which allows us to write the general formulas above for ${\rm Cr}_{2k, 2m}$. We show the calculations below. We have
\begin{eqnarray}
{\rm Cr}_{6, 6}  & \ = \ & \left(6-1\right)!! - {\rm Cr}_{6, 4} - {\rm Cr}_{6, 2} - {\rm Cr}_{6, 0} \nonumber \\ & \ = \ & 5!! - \binom{6}{1} - 0 - C_{3} \ = \ 15 - 6 - 0 - 5 \ = \ 4,
\end{eqnarray}
so ${\rm Cr}_{2k, 6} = 4\binom{2k}{k-3}$,
and thus
\bea
{\rm Cr}_{8, 8} & \ = \ & \left(8-1\right)!! - {\rm Cr}_{8, 6} - {\rm Cr}_{8, 4} - {\rm Cr}_{8, 2} - {\rm Cr}_{8, 0} \nonumber \\
& \ = \ & 7!! - 4\binom{8}{1} - \binom{8}{2} - 0 - C_{4} \ = \ 105 - 32 - 28 - 14 \ = \ 31.
\eea

To finish the calculation for ${\rm Cr}_{2k, 8}$ we compute
\be
\sum_{0 < a < 4} {\rm Cr}_{2a, 2a} {\rm Cr}_{8-2a, 8-2a} \ = \ {\rm Cr}_{2, 2} {\rm Cr}_{6, 6} + {\rm Cr}_{4, 4} {\rm Cr}_{4, 4} + {\rm Cr}_{6, 6} {\rm Cr}_{2,2} \ = \ 0 + 1 + 0 \ = \ 1.
\ee
so that we get ${\rm Cr}_{2k, 8} \ = \  31\binom{2k}{k-4}+\sum_{d=1}^{k-4}\binom{2k}{k-4-d}\left(4+d\right)$.

For the formula for ${\rm Cr}_{2k, 10}$,
\bea
{\rm Cr}_{10, 10} & \ = \ & \left(10-1\right)!! - {\rm Cr}_{10, 8} - {\rm Cr}_{10, 6} - {\rm Cr}_{10, 4} - {\rm Cr}_{10, 2} - {\rm Cr}_{10, 0} \nonumber \\
& \ = \ & 9!! - \left(31\binom{10}{1} + \sum_{d = 1}^{1}\binom{10}{1-d}\left(4+d\right)\right) - 4\binom{10}{2} - \binom{10}{3}-0-C_{5} \nonumber \\
& \ = \ & 945 - (310 + 5) - 4\left(45\right) - 120 - 0 - 42 \ = \ 288,
\eea
and finally
\bea
\sum_{0 < a < 5} {\rm Cr}_{2a, 2a} {\rm Cr}_{10-2a, 10-2a} & \ = \ & {\rm Cr}_{2, 2}{\rm Cr}_{8, 8} + {\rm Cr}_{4, 4} {\rm Cr}_{6, 6} + {\rm Cr}_{6, 6} {\rm Cr}_{4, 4} + {\rm Cr}_{8, 8}{\rm Cr}_{2, 2} \nonumber\\ & \ = \ & 0 + 4 + 4 + 0 \ = \ 8,
\eea
so ${\rm Cr}_{2k, 10}  \ = \  288\binom{2k}{k-5}+8\sum_{d=1}^{k-5}\binom{2k}{k-5-d}\left(5+d\right)$.
\end{proof}

Notice that by using the formulas in Lemma \ref{lem:crossingformulas} to calculate the number of terms with each of the possible contributions given in Theorem \ref{thm:weightedcontributions}, we are able to calculate up to the $12^{\text{th}}$ moment exactly (where for the $12^{\text{th}}$ moment we use the same recursive procedure as in the proof of Lemma \ref{lem:crossingformulas} to calculate ${\rm Cr}_{12, 12}$).

\begin{rek}\label{eq:kl1coeff} The coefficients in front of the binomial coefficient of the leading term of ${\rm Cr}_{2k,2m}$ are sequence A081054 from the OEIS \cite{Kl1}.
\end{rek}


\section{Limiting Behavior of the Moments}\label{sec:limiting behavior}

As we are unable to find exact expressions for the number of pairings with exactly $2m$ crossing vertices for all $m$, we determine the expected value and variance of the number of vertices in a crossing. Such expressions, and the limiting behavior of these expressions, are useful for obtaining bounds for the moments. To find these, we make frequent use of arguments about the probabilities of certain pairings, recognizing that since all configurations are equally likely, the probability that a vertex $i$ pairs with a vertex $j$ is just $\frac{1}{2k-1}$.

\begin{thm}\label{thm:evcrossing}
The expected number of vertices involved in a crossing of $2k$ vertices paired on the circle is \be\label{eq:pcrossexact} \frac{2k}{2k-1}\left(2k-2 - \frac{{\ }_2F_1(1,3/2,5/2-k;-1)}{2k-3} - (2k-1){\ }_2F_1(1,1/2+k,3/2;-1)\right), \ee which is
\be\label{eq:pcrossaymptoticnew} 2k-2 - \frac2{k} + O\left(\frac1{k^2}\right)\ee as $k \rightarrow \infty$.
\end{thm}
\begin{proof} In our main applications (such as computing the asymptotic behavior of the mean and the variance), we only need the asymptotic expression \eqref{eq:pcrossaymptoticnew}, which we prove elementarily below. We give the proof of \eqref{eq:pcrossexact} in Appendix \ref{sec:appendixpcross}, which involves converting the expansions below to differences of hypergeometric series.

For a given pairing of $2k$ vertices, let $X_{i} = 1$ if vertex $i$ is involved in a crossing and $0$ otherwise. Then $Y_{2k} = \sum_{i=1}^{2k}X_{i}$ is the number of vertices involved in a crossing in this pairing. By linearity of expectation, \be
\E\left(Y_{2k}\right) \ = \ \E\left(\sum_{i=1}^{2k}X_{i}\right)  \ = \ 2k\E\left(X_{i}\right) \ = \ 2 k p_{\rm cross},
\ee
where $p_{\rm cross}$ is the probability that a given vertex is in a crossing as, by symmetry, this is the same for all vertices. Thus, without loss of generality, we may think of $p_{\rm cross}$ as the probability that vertex $1$ is in a crossing. We notice that

\begin{enumerate}

\item If vertex 1 is matched with another odd indexed vertex, which happens with probability $\frac{k-1}{2k-1}$, then it must be involved in a crossing, since there are an odd number of vertices in the two regions created by the matching, meaning that the regions cannot only pair by themselves.

\item If vertex 1 is matched with an even indexed vertex, then it is involved in a crossing if and only if it does not partition the remaining vertices into two parts that pair exclusively with themselves. Suppose it is matched with vertex $2m$ (which happens with probability $\frac{1}{2k-1}$). Then its edge divides the vertices into a region of $2m-2$ and a region of $2k-2m$ vertices. As the number of ways to match $2\ell$ objects in pairs with order immaterial is $(2\ell-1)!!$ $=$ $(2\ell-1)$ $(2\ell-3)$ $\cdots$ $3 \cdot 1$, the probability that each region pairs only with itself is
\begin{equation}
\frac{\left(2m-3\right)!!\left(2k-2m-1\right)!!}{\left(2k-3\right)!!}.
\end{equation}

\end{enumerate}

Thus, the probability that vertex 1 is involved in a crossing is
\bea\label{eq:expansionforpcross}
p_{\rm cross} & \ = \ & \frac{k-1}{2k-1}+\sum_{m=2}^{k-1}\frac{1}{2k-1}\left(1-\frac{\left(2m-3\right)!!\left(2k-2m-1\right)!!}{\left(2k-3\right)!!}\right) \nonumber\\
& \ = \ & \frac{2k-3}{2k-1}- \frac{1}{2k-1}\sum_{m=2}^{k-1}\frac{\left(2m-3\right)!!\left(2k-2m-1\right)!!}{\left(2k-3\right)!!} \nonumber\\
& \ = \ & \frac{2k-3}{2k-1} - \frac{1}{2k-1}\sum_{m=2}^{k-1}\frac{\left(2m-3\right)! \left(2k-2m\right)! \left(2k-4\right)!!}{\left(2m-4\right)!!\left(2k-2m\right)!!\left(2k-3\right)!} \nonumber\\
& \ = \ & \frac{2k-3}{2k-1} -  \frac{1}{2k-1}\sum_{m=2}^{k-1}\frac{\left(2m-3\right)!\left(2k-2m\right)!2^{k-2}\left(k-2\right)!}{2^{m-2}\left(m-2\right)!2^{k-m}\left(k-m\right)!\left(2k-3\right)!} \nonumber\\
& \ = \ & \frac{2k-3}{2k-1} - \frac{1}{2k-1}\sum_{m=2}^{k-1} \frac{\binom{k-2}{m-2}}{\binom{2k-3}{2m-3}}.
\eea
Therefore
\be\label{eq:complicatedsumformeanpcross}
\E\left(Y_{2k}\right) \ = \ 2kp_{\rm cross} \ = \ \left(2k\right)\frac{2k-3}{2k-1}-\left(2k\right)\frac{1}{2k-1}\sum_{m=2}^{k-1} \frac{\binom{k-2}{m-2}}{\binom{2k-3}{2m-3}}.
\ee
In the above sum, the first and last terms are both $\frac1{2k-3}$, as for $m=2$ we have \begin{equation}
\frac{\binom{k-2}{0}}{\binom{2k-3}{1}}\ =\ \frac{1}{2k-3},
\end{equation} and for $m=k-1$ we have
\begin{equation}
\frac{\binom{k-2}{k-3}}{\binom{2k-3}{2k-5}}\ =\ \frac{\binom{k-2}{1}}{\binom{2k-3}{2}}\ = \ \frac{2\left(k-2\right)}{\left(2k-3\right)\left(2k-4\right)}\ =\ \frac{1}{2k-3}.
\end{equation} Looking at the ratio of subsequent terms, straightforward algebra shows \be \frac{\ncr{k-2}{m-1}/\ncr{2k-3}{2m-1}}{\ncr{k-2}{m-2}/\ncr{2k-3}{2m-3}} \ = \ \frac{2m-1}{2k-2m-1}. \ee Thus for $m$ up to the halfway point, each term in the sum is less than the previous. In particular, the $m=3$ term is $5/(2k-7)$ times the $m=2$ term, and hence all of these terms are $O(1/k^2)$. Similarly, working from $m=k-2$ to the middle we find all of these terms are also $O(1/k^2)$, and thus the sum in \eqref{eq:complicatedsumformeanpcross} can be rewritten, giving
\bea\label{eq:finalmeanY2k}
\E\left(Y_{2k}\right) & \ = \ & \left(2k\right)\frac{2k-3}{2k-1}-\left(2k\right)\frac{1}{2k-1}\left(\frac{2}{2k-3} + O\left(\frac{1}{k^{2}}\right)\right) \nonumber\\
& \ = \ & 2k - 2 - \frac{2}{k} + O\left(\frac{1}{k^{2}}\right).
\eea
\end{proof}



\begin{thm}\label{thm:varcrossing}
The variance of the number of vertices involved in a crossing approaches $4$ as $k \rightarrow \infty$.
\end{thm}
\begin{proof}
We need to calculate ${\rm Var}\left(Y_{2k}\right) = \E\left(Y_{2k}^{2}\right)- \E\left(Y_{2k}\right)^{2}$. As we know the second term by Theorem \ref{thm:evcrossing}, we concentrate on the first term:
\be
\mathbb{E}\left(Y_{2k}^{2}\right) \ = \ \sum_{i, j \in \{1, \dots, 2k\}}\mathbb{E}\left(X_{i}X_{j}\right).
\ee
The above sum has $4k^{2}$ terms.

For $2k$ of those terms, $i = j$ so $\mathbb{E}\left(X_{i}X_{j}\right) = \mathbb{E}\left(X_{i}^{2}\right) = \mathbb{E}\left(X_{i}\right) = p_{\rm cross}$ as the $X_\ell$'s are binary indicator variables with probability of success $p_{\rm cross}$. For another $2k$ terms, we have $i$ and $j$ are paired on the same edge, so $\mathbb{E}\left(X_{i}X_{j}\right) = \mathbb{E}\left(X_{i}\right) = p_{\rm cross}$ as before.

For the remaining $4k^{2}-4k$ terms, $i$ and $j$ are on different edges, and we must find the probability that both those edges are in crossings. We separate this probability into two disjoint probabilities, the probability $p_{a}$ that they cross each other, and the probability that they don't cross each other but are each crossed by at least one other pairing. We denote this second probability by $\left(1-p_{a}\right)p_{b}$, where $p_{b}$ is the conditional probability they are each crossing given that they don't cross each other. We will find these probabilities by taking sums over the placements of $k, m, p, q$ above as appropriate and calculating for each the probability of observing one of our desired configurations. We have shown
\be\label{eq:varwithpabc}
\mathbb{E}\left(Y_{2k}^{2}\right)  \ = \  4kp_{\rm cross} + \left(4k^{2} - 4k\right) \left(p_{a}+\left(1-p_{a}\right)p_{b}\right),
\ee thus reducing the problem to the determination of $p_{a}$ and $p_{b}$.

Without loss of generality, we label our edges as $\left\{1, m\right\}$ and $\left\{p, q\right\}$. They cross each other if and only if one of $\left\lbrace p, q \right\rbrace$ is one of the $m-2$ vertices between $1$ and $m$, and the other is one of the $2k-m$ vertices between $m$ and $2k$. Thus
\bea
p_{a} & \ = \ & \sum_{m = 2}^{2k} \frac{1}{2k-1} \cdot 2 \cdot \frac{m-2}{2k-2} \cdot \frac{2k-m}{2k-3} \nonumber\\
& \ = \ & \frac{2}{\left(2k-1\right)\left(2k-2\right)\left(2k-3\right)} \left[\sum_{m=2}^{2k}-4k - \sum_{m=2}^{2k}m^{2} + \left(2k+2\right)\sum_{m=2}^{2k}m\right].
\eea

By using the formulas for the sum of the first $n$ integers and the first $n$ squares, we simplify the second factor to
\be
\left(2k-1\right)\left(-4k\right)-\left(\frac{2k\left(2k+1\right)\left(4k+1\right)}{6}-1\right)+ \left(2k+2\right)\left(\frac{2k\left(2k+1\right)}{2}-1\right),
\ee
which gives
\bea
p_{a}
& \ = \ & \frac{2}{\left(2k-1\right)\left(2k-2\right)\left(2k-3\right)} \frac{\left(2k-1\right)\left(2k-2\right)\left(2k-3\right)}{6} \ = \ \frac13.
\eea

We now calculate $p_{b}$, the probability that $\left\{1, m\right\}$ and $\left\{p, q\right\}$ are both involved in crossings given they don't cross each other. We must place $\left\{1, m\right\}, \left\{p, q\right\}$. Relabeling if necessary, we may assume $1 < m < p < q$; such a labeling is possible if and only if $\left\{1, m\right\}$ and $\left\{p, q\right\}$ do not cross each other. We compute the complement of our desired probability by finding the number of configurations where one or less of $\left\{1, m\right\}$ and $\left\{p, q\right\}$ is in a crossing. We denote the number of such configurations by $N_{k, m, p, q}$ and can thus write

\begin{equation}\label{eq:pbsumwithN}
p_{b} \ = \ 1 - \sum_{m=2}^{2k-2}\sum_{p=m+1}^{2k-1}\sum_{q = p+1}^{2k} \frac{N_{k, m, p,q}}{\left(2k-5\right)!!}.
\end{equation}

Since there are $\binom{2k-1}{3}$ terms in the above sum (corresponding to the $\binom{2k-1}{3}$ possible choices of $m, p, q$ since we have specified the location of vertex $1$ and the order of $m, p, q$), we can rewrite \eqref{eq:pbsumwithN} as
\begin{equation}
p_{b} \ = \ 1 - \frac{\sum_{m=2}^{2k-2}\sum_{p=m+1}^{2k-1}\sum_{q = p+1}^{2k} N_{k, m, p,q}}{\binom{2k-1}{3}\left(2k-5\right)!!}.
\end{equation}

All that remains to be done is to evaluate the sum in the above expression. To do so, we first define the following function $P\left(k\right)$, which counts the number of ways $k$ vertices can be paired with each other:
\begin{equation}
\threecase{P(x)\ =\ }{0}{if $k$ is odd}{1}{if $k=0$}{(k-1)!!}{otherwise.}
\end{equation}

Next we think of these two edges as dividing the remaining vertices into three regions: those between $\left\{1, m\right\}$ and $\left\{p, q\right\}$, of which there are $M = p-m-1+2k-q$, those on the side of $\left\{1, m\right\}$, of which there are $L = m-2$, and those on the side of $\left\{p, q\right\}$, of which there are $R = q-p-1$. We know that $\left\{1, m\right\}$ will not be crossed if the $L$ vertices between $1$ and $m$ pair exclusively with each other. Likewise, $\left\{p, q\right\}$ will not be crossed if the vertices between $p$ and $q$ pair exclusively with each other. Our desired quantity is thus the union of these two events less their intersection:
\begin{equation}
P\left(L+M\right)P\left(R\right)+P\left(R+M\right)P\left(L\right)-P\left(L\right)P\left(M\right)P\left(R\right).
\end{equation}

Notice that if $L$ or $R$ is $0$, one of $\left\{1, m\right\}, \left\{p, q\right\}$ is an adjacent edge, and therefore is not crossing. Thus
\begin{equation}\label{eq:Nkmpq}
\twocase{N_{k, m, p, q} \ = \ }{\left(2k-5\right)!!}{if $L$ or $R$ is 0}{P\left(L+M\right)P\left(R\right)+P\left(R+M\right)P\left(L\right)- P\left(L\right)P\left(M\right)P\left(R\right)}{otherwise.}
\end{equation}

We now investigate the limiting behavior of $p_{b}$ (given in \eqref{eq:pbsumwithN}) by using the cases in \eqref{eq:Nkmpq}.

\begin{itemize}

\item For the first case, we have $L$ or $R$ is zero, and thus $N_{k, m, p, q} = (2k-5)!!$. We are reduced to counting the number of terms with $L$ or $R$ zero. Note that $L = 0$ when $m = 2$, and $R = 0$ when $q = p + 1$. Each of these events happens in $\binom{2k-2}{2}$ pairings (we have fixed either $m$ or $q$, and the other $2$ vertices are chosen from the remaining $2k-2$ vertices), and their intersection is $\binom{2k-3}{1}$ ($p$ is the only free index) pairings. In the limit, this case contributes
\be
\frac{\left(2\binom{2k-2}{2}-\binom{2k-3}{1}\right)\left(2k-5\right)!!}{\binom{2k-1}{3}\left(2k-5\right)!!} \ = \ \frac{3}{k} + O\left(\frac{1}{k^{3}}\right).
\ee

\item For the second case, $L$ and $R$ are non-zero. We first evaluate the contribution of the first two terms (notice that they will contribute the same in the sum  since you can simply relabel $\left\{1, m\right\}$ and $\left\{p, q\right\}$) and then the third term, recalling that we only have to look for terms that are at least $O\left(\frac{1}{k^{2}}\right)$ since we can see in \eqref{eq:varwithpabc} that any other terms will not contribute in the limit as $k \rightarrow \infty$.

\begin{itemize}
\item For $P\left(L+M\right)P\left(R\right)$, the largest terms are from when either $L+M = 2$, or when $R = 2$. In these cases, $N_{k, m, p, q} = \left(2k-7\right)!!$. If $R = 2$ then $q = p + 3$ and $m, p$ are free so there are $\binom{2k-4}{2}$ such terms corresponding to the $\binom{2k-4}{2}$ choices of $m$ and $p$. If $L+M = 2$ and $L \neq 0$ then there are only two possible terms: either $L = 1, M = 1, R = 2k-6$ or $L = 2, M = 0, R = 2k-6$. Including the symmetric terms for $P\left(R+M\right)P\left(L\right)$, these terms thus have a combined contribution of
\be
\frac{2\left(\binom{2k-4}{2}+2\right)\left(2k-7\right)!!}{\binom{2k-1}{3}\left(2k-5\right)!!} \ = \ \frac{3}{2k^{2}} + O\left(\frac{1}{k^{3}}\right).
\ee

\item For the third term, $-P\left(L\right)P\left(M\right)P\left(R\right)$, the largest contributions are when two regions combine for exactly $2$ vertices, which gives a contribution of $\left(2k-7\right)!!$. If we disregard the requirement that $L$ and $R$ are nonzero in order to obtain an upper bound on the magnitude of this contribution, there are $3$ possible terms. The next largest contribution will be when two regions combine for exactly $4$ vertices which gives a contribution of $\left(2k-9\right)!!$. Proceeding with these diagonal terms, we know that the third term contributes at most in magnitude
\be \ \ \ \ \ \ \ \ \  \ \ \ \ \ \ \ \ \ \ \ \ \ \ \  \ \ \ \ \ \ \ \ \ \ \ \ \
3\frac{\left(2k-7\right)!!}{\binom{2k-1}{3}\left(2k-5\right)!!}\ + \ 6\frac{\left(2k-9\right)!!}{\binom{2k-1}{3}\left(2k-5\right)!!}\ +\ 9\frac{\left(2k-11\right)!!}{\binom{2k-1}{3}\left(2k-5\right)!!}\ +\ \cdots
\ =\ O\left(\frac{1}{k^{3}}\right),
\ee
so they do not contribute to the main term in the limit.

\end{itemize}
\end{itemize}

Thus we have that, as $k \rightarrow \infty$,
\be
p_{b} \ = \ 1 - \frac{3}{k} - \frac{3}{2k^{2}} + O\left(\frac{1}{k^{3}}\right).
\ee
Therefore if we substitute for $p_{a}$ and $p_{b}$ in \eqref{eq:varwithpabc} we find
\bea
\E\left(Y_{2k}^{2}\right) & \ = \ & 4k-4 + \left(4k^{2} - 4k\right) \left(\frac{1}{3}+\frac{2}{3}\left(1-\frac{3}{k} - \frac{3}{2k^{2}}\right)\right) \\
& \ = \ & 4k^{2}-8k + O\left(\frac{1}{k}\right).
\eea

Using \eqref{eq:finalmeanY2k}, we also have that
\be
\E\left(Y_{2k}\right)^{2} \ = \ \left(2k - 2 - \frac{2}{k} + O\left(\frac{1}{k^{2}}\right)\right)^{2} \ = \ 4k^{2} - 8k - 4 + O\left(\frac{1}{k}\right).
\ee
The variance is $\E\left(Y_{2k}^{2}\right)-\E\left(Y_{2k}\right)^{2}$, which is $4 + O(1/k)$ as $k \rightarrow \infty$.
\end{proof}

Figure \ref{fig:meanvarsims} provides a numerical verification of the above formulas for the expected values and variances.

\begin{center}
\begin{figure}
\includegraphics[height=30mm]{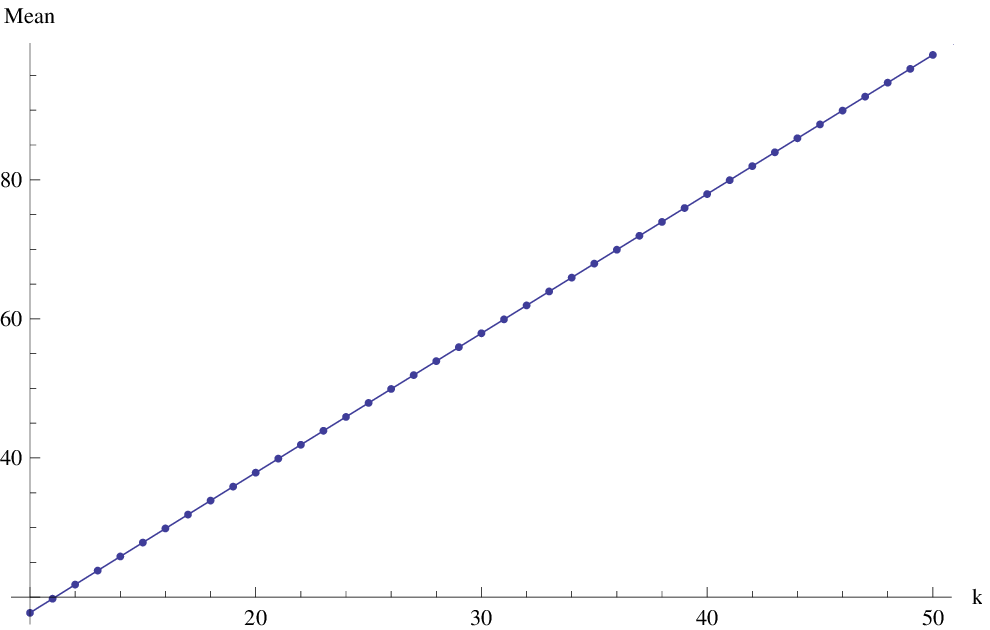}\ \ \includegraphics[height=30mm]{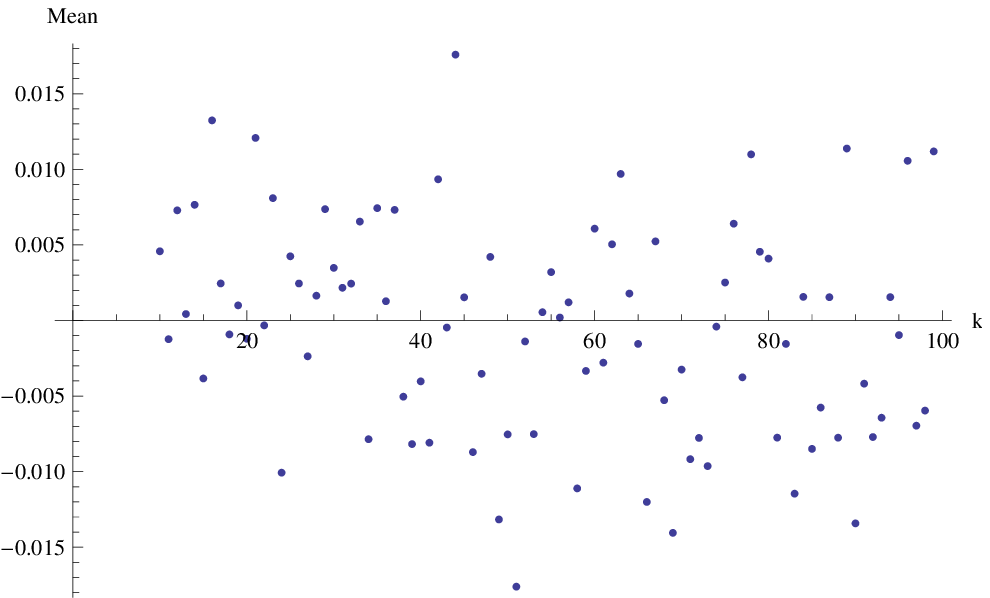}\ \
\includegraphics[height=30mm]{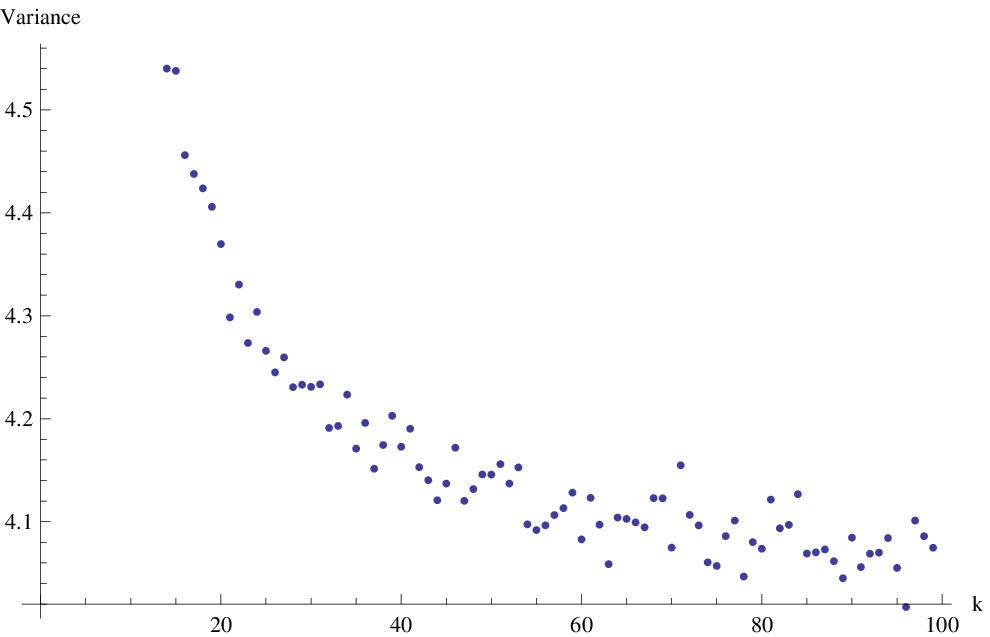}
\caption{\label{fig:meanvarsims} Numerical confirmation of formulas for the expected value and variance of vertices involved in crossing. The first plot is the expected value for $2k$ vertices (solid line is theory) versus $k$, the second plot is a plot of the deviations from theory, and the third plot is the observed variance; all plots are from 100,000 randomly chosen matchings of $2k$ vertices in pairs.}
\end{figure}
\end{center}


\section{Limiting Spectral Measure}\label{sec:limitingspectralmeasure}

We now complete the proof of Theorem \ref{thm:existencelimspecmeas} by showing convergence and determining the support.

\begin{proof}[Proof of Theorem \ref{thm:existencelimspecmeas}] The proof of the claimed convergence is standard, and follows immediately from similar arguments in \cite{HM,MMS,JMP,KKMSX}. Those arguments rely only on degree of freedom counting arguments, and are thus applicable here as well. We are left with determining the limiting rescaled spectral measures.

\begin{itemize}

\item $p = 1/2$: If $p = 1/2$, we know from \eqref{thm:weightedcontributions} that only those configurations with no crossings contribute. In particular, we may apply this to the $\pm 1$ real symmetric weight matrix. Moreover, in the crossing configurations, it is simple to check that in $(N - n\cdot o(N))^n \approx N^n - o(N^n)$ of the $N^n$ terms for the $n$th moment computation each random variable from the coefficients of the matrix ensemble occurs exactly twice. Since the moments of the original distribution are finite, the remaining $o(N^n)$ terms do not contribute. Thus, we may assume that each random variable occurs exactly twice in each term (and the variables are otherwise independent.) The claim follows directly from recalling that the number of non-crossing configurations are simply the Catalan numbers (see for example \cite{Fo}), which are also the moments of the semi-circle distribution.\\

\item $p > 1/2$: We consider the case when $\widetilde{\mu}$ (the limiting rescaled spectral measure) has unbounded support; the case of bounded support is similar. To show that the limiting signed rescaled spectral measure has unbounded support it suffices to show that the moments of our distribution grow faster than any exponential bound, i.e., that for all $B$ there exists some $k$ such that $M_{2k} > B^{2k}$. Assume the moments of the unsigned ensemble grow faster than exponentially. We prove that our distribution similarly has unbounded support using this fact and by considering the ``worst-case'' scenario allowed for under Theorem \ref{thm:weightedcontributions}. Namely, we suppose that each term contributes $x(c) \left(2p-1\right)^{2k}$, which gives us the smallest moment possible. In this case, $M_{2k}$ is decreased from the unsigned case by a factor of $(2p-1)^{2k}$, and thus the growth is still faster than any exponential bound.

\end{itemize}

\end{proof}


\section{Concluding Remarks and Future Work}

In our analysis of the limiting rescaled spectral measure of signed structured ensembles, it was crucial each random variable occurs $o(N)$ times in each row of matrices in the ensemble and that the original structured ensemble have its empirical rescaled spectral measures converge to a limiting measure; we plan to revisit cases where these assumptions fail (such as the examples in Remark \ref{rek:blocks}) in a sequel paper. The key to our analysis is Remark \ref{rek:keyforcontributions}. The moments of the signed ensemble are depressed by a factor which depends on the structure of the matrices. If $p = 1/2$ then there is no contribution from the configurations with crossings (as their contribution is at most $(2p-1)^2$ times its unweighted value). This leaves the contribution from the non-crossing configurations. These are governed by the Catalan numbers; as everything is matched in pairs without crossing, each of these configurations gives 1 and we regain the semi-circle. The computations become more involved and more dependent on the structure for $p \in (1/2, 1]$, as the structure of the matrix can force repeated indices in the product of the weights, which of course affects its expected value and contribution.

In addition to obtaining limiting measures for signed structured ensembles, we isolate some combinatorial results which are related to issues in knot theory (such as Theorem \ref{thm:momentformulas}). We also obtain asymptotics for the number of pairings of $2k$ vertices with exactly $2m$ crossing vertices. While we can derive a closed form expression for the expected number (Theorem \ref{thm:evcrossing}), the formula for the variance is more involved and we content ourselves here with determining its asymptotic, and a natural future project is to see if explicit formulas for the higher moments of the number of pairings with a given number of crossing vertices exist (or, even better, to see if a nice distribution governs the behavior as $k$ and $m$ tend to infinity).

\appendix


\section{Exact formula for mean number of crossings}\label{sec:appendixpcross}

To prove \eqref{eq:pcrossexact}, it suffices to simplify the sum in the expansion of $p_{\rm cross}$ in \eqref{eq:expansionforpcross}. We first extend the $m$ sum to include $m=k$; this adds 1 to the sum which must then be subtracted from the term outside. For notational convenience, set $n = k-2$. We re-index and let $m$ run from $0$ to $n$, and are thus reduced to analyzing \be S(n) \ = \ \sum_{m=0}^n \frac{\ncr{n}{m}}{\ncr{2n+1}{2m+1}}.\ee The following notation and properties are standard (see for example \cite{GR}). The Pochhammer symbol $(x)_m$ is defined for $m \ge 0$ by \be(x)_m \ = \ \frac{\Gamma(x+m)}{\Gamma(x)} \ = \ x(x+1) \cdots (x+m-1),\ee and the hypergeometric function ${\ }_2F_1$ by \be {\ }_2F_1(a,b,c;z) \ = \ \sum_{m=0}^\infty \frac{(a)_m(b)_m}{(c)_m} \frac{z^m}{m!}, \ee which converges for all $|z| < 1$ so long as $c$ is not a negative integer.

For ease of exposition, we work backwards from the answer.\footnote{Mathematica is able to evaluate such sums and suggest the correct hypergeometric combinations. One has to be a little careful, though, as Mathematica incorrectly evaluated $S(n)$, erroneously stating that there was zero contribution if we extend the sum to all $m$. In other words, it thought $S(n) = T_1(n) = T_1(n)+T_2(n)$ in the notation introduced below.} Using $\Gamma(1+z) = z\Gamma(z)$ and $\Gamma(1+\ell) = \ell!$ (for integral $\ell$), we find  \bea {\ }_2F_1(1,3/2,1/2-n,-1) & \ = \ & \sum_{m=0}^\infty\frac{(1)_m (3/2)_m}{(1/2-n)_m} \frac{(-1)^m}{m!} \nonumber\\ &=& \sum_{m=0}^\infty \frac{\Gamma(1+m)}{\Gamma(1)} \frac{\Gamma(3/2+m)}{\Gamma(3/2)} \frac{\Gamma(1/2-n)}{\Gamma(1/2-n+m)} \frac{(-1)^m}{m!} \nonumber\\ &=: & T_1(n) + T_2(n), \eea where $T_1(n)$ is the sum over $m\le n$ and $T_2(n)$ is the sum over $m > n$. From the functional equation of the Gamma function and using $\ell!! = \ell(\ell-2)(\ell-4)\cdots$ down to 2 or 1, we find \bea \Gamma(3/2+m) & \ = \ & 2^m (2m+1)!! \Gamma(3/2) \nonumber\\ \Gamma(1/2-n+m) &=& (-1)^m 2^m (2n-1)(2n-3)\cdots(2n-2m+1) \Gamma(1/2-n). \eea Substituting, we find \bea T_1(n) & \ = \ & \sum_{m=0}^n \frac{(2m+1)!! (2n-2m-1)!!}{(2n-1)!!} \nonumber\\  & \ = \ & \sum_{m=0}^n \frac{(2m+1)! (2n-2m-1)!}{(2n-1)! 2n} \frac{2n(2n-2)!!}{(2m)!!}{(2n-2m-2)!!}
\nonumber\\  & \ = \ & \sum_{m=0}^n \frac{(2m+1)!(2n-2m)!}{(2n+1)!} \cdot (2n+1) \cdot \frac{2^n n!}{(2n-2m) 2^{n-1} m!(n-m-1)!} \nonumber\\ &=& (2n+1) \sum_{m=0}^n \frac{\ncr{n}{m}}{\ncr{2n+1}{2m+1}}; \eea note this is our desired sum. Thus \be  \sum_{m=0}^n \frac{\ncr{n}{m}}{\ncr{2n+1}{2m+1}} \ = \ \frac{{\ }_2F_1(1,3/2,1/2-n,-1) - T_2(n)}{2n+1}, \ee and the proof is completed by analyzing $T_2(n)$. To determine this term's contribution, we re-index. Writing $m = n+1+u$, we find \bea T_2(n) & \ = \ & \sum_{u=0}^\infty \frac{\Gamma(1+n+1+u)}{\Gamma(1)} \frac{\Gamma(3/2+n+1+u)}{\Gamma(3/2)} \frac{\Gamma(1/2-n)}{\Gamma(1/2-n+n+1+u)} \frac{(-1)^{n+1+u}}{(n+1+u)!} \frac{u!}{u!} \nonumber\\ &=& \sum_{u=0}^\infty \frac{\Gamma(1+u)}{\Gamma(1)} \frac{\Gamma(5/2+n+u)}{\Gamma(3/2)} \frac{\Gamma(1/2-n)}{\Gamma(3/2+u)} \frac{(-1)^{n+1} (-1)^u}{u!} \nonumber\\ &=& \frac{(-1)^{n+1} \Gamma(1/2-n)\Gamma(5/2+n)}{\Gamma(3/2)^2} \sum_{u=0}^\infty \frac{\Gamma(1+u)}{\Gamma(1)} \frac{\Gamma(5/2+n+u)}{\Gamma(5/2+n)} \frac{\Gamma(3/2)}{\Gamma(3/2+u)} \frac{(-1)^u}{u!} \nonumber\\ &=& -(2n+3)(2n+1){\ }_2F_1(1,1/2+k,3/2,-1), \eea where we used $\Gamma(1-z)\Gamma(z) = \pi/\sin(\pi z)$ with $z= n + \frac12$ to simplify the Gamma factors depending only on $n$. Combining the above proves \eqref{eq:pcrossexact}.


\ \\

\end{document}